\documentclass[11pt]{article}

\usepackage{amsmath,amsthm,amssymb,amsfonts,multicol,mathrsfs,enumerate}

\usepackage{tikz}
\usetikzlibrary{matrix,arrows}

\usepackage{lmodern}

\DeclareFontFamily{U}{rsfs}{\skewchar\font127}
\DeclareFontShape{U}{rsfs}{m}{n}{%
  <-6> rsfs5
  <6-8> rsfs7
  <8-> rsfs10
}{}

\renewcommand{\O}{\mathcal{O}}

\newcommand{\C}{\mathbb{C}}
\newcommand{\F}{\mathbb{F}}
\renewcommand{\P}{\mathbb{P}}
\newcommand{\A}{\mathbb{A}}

\renewcommand{\bar}{\overline}

\renewcommand{\i}{^{-1}}

\newcommand{\Aut}{\mbox{Aut}}

\newcommand{\Spec}{\mbox{\small{Spec }}}

\newcommand{\Ext}{\mbox{\small{Ext}}}

\newcommand{\ra}{\rightarrow}
\newcommand{\sub}{\subset}

\newcommand{\isom}{\cong}
\newcommand{\iso}{\cong}

\newcommand{\inj}{\hookrightarrow}
\newcommand{\surj}{\twoheadrightarrow}

\newtheoremstyle{natheorem}
  {}
  {}
  {\addtolength{\leftskip}{1em} \normalsize}
  {0em}
  {\bfseries \normalsize}
  {.}
  {.5em}
  {}

\theoremstyle{natheorem}
\newtheorem{prop}{Proposition}
\newtheorem{thm}[prop]{Theorem}
\newtheorem{cor}[prop]{Corollary}

\newtheorem{defn}[prop]{Definition}

\newtheorem{lemma}[prop]{Lemma}
\newtheorem*{prb*}{Problem}
\newtheorem{construction}[prop]{Construction}

\usepackage{array}

\newcommand{\Sym}{\mbox{\small\textbf{Sym}}}
\newcommand{\Bun}{\mbox{\small\textbf{Bun}}}
\newcommand{\dlog}{\mbox{\small\textbf{dlog}}}

\renewcommand{\H}{\mathscr{H}}

\renewcommand{\O}{\mathcal{O}}
\newcommand{\U}{\mathcal{U}}
\newcommand{\END}{\mathcal{E}nd}
\newcommand{\HOM}{\mathcal{H}om}
\newcommand{\id}{\mbox{id}}
\renewcommand{\S}{\mathcal{S}}
\renewcommand{\sl}{\mathfrak{sl}}
\renewcommand{\phi}{\varphi}
\renewcommand{\H}{\mathbb{H}}
\renewcommand{\part}{\partial}
\renewcommand{\C}{\mathscr{C}}

\newcommand{\Tor}{\mbox{Tor}}
\newcommand{\E}{\mathbb{E}}
\newcommand{\Res}{\mbox{\small\textbf{Res}}}
\renewcommand{\C}{\mathbb{C}}

\newcommand{\Hyp}{\mbox{\textbf{H}}}
\newcommand{\Di}{\Delta_\infty}
\renewcommand{\F}{\mathcal{F}}

\usepackage{multirow}

\title{Isomorphisms between moduli of parabolic Higgs bundles}
\author{Nathan Clement}
\begin{document}
\maketitle
\begin{abstract}
  In this paper we study four families of moduli problems which give rise to two dimensional examples of the Hitchin map.  
  Using a few Fourier-Mukai transforms on the corresponding spectral curves, we give isomorphisms between these moduli problems.
\end{abstract}
\setcounter{tocdepth}{1}
\tableofcontents
\section{Introduction}
In 1987, Nigel Hitchin studied the moduli spaces of what are now called Higgs bundles in his paper \textit{Stable bundles and Integrable Systems}.  
A Higgs bundle on a fixed Riemann surface $X$ is a vector bundle $V$ together with a ``Higgs field'' which is nothing but an $\O_X$ linear map $\phi:V\ra \Omega_X\otimes V$.
Such an object encodes a cotangent vector to the moduli space of vector bundles on $X$ at the point $[V]$.
This computation is not hard to check:
a tangent vector at $V$ is a vector bundle $\tilde{V}$ on $X\times \Spec \C[\epsilon]/\epsilon^2$ that extends the bundle $V$ on the closed subscheme $X$.
Such a thing is called a first order deformation of $V$, and a simple argument shows the first order deformations of $V$ are given by $\Ext^1(V,V)\isom H^1(X,\END (V))$. 
Since the bundle $\END (V)$ is self dual, Serre duality tells us that a dual vector is a global section of $\Omega_X\otimes\END (V)$.
Such a ``twisted endomorphism'' on a bundle has a ``spectrum'' which consists of a curve $S$ (the ``spectral curve'') 
embedded in the total space of the twisting line bundle $|\Omega|$ over $X$.
The pair $(V,\phi)$ may be studied as a module $M_\phi$ on the total space of this line bundle, with the curve $S$ a sort of upper bound on its support.
One of the amazing features of Hitchin's paper is that he finds that the possible spectral curves form an affine space of dimension half that of the whole moduli space.
The map to this affine space is now called the Hitchin map and has deep significance for the symplectic geometry of the space.

In this paper we study vector bundles with ``logarithmic'' Higgs fields.
These objects are the same as Higgs bundles except that the endomorphism is twisted not by the line bundle of differentials, 
but by the line bundle of differentials allowed first order poles at some specified points of $X$.
In fact, the Riemann surface in this paper is always $\P^1$, which has no good moduli spaces of vector bundles.
Instead, the logarithmic Higgs fields studied in this paper are objects of a (twisted) cotangent space of what are called parabolic vector bundles on $\P^1$.
In this paper, ``parabolic bundle''
\footnote{As opposed to a vector bundle with a (partial) filtration by sub-bundle. Some authors have called these fiberwise flags ``quasi-parabolic''.}
will mean a vector bundle together with partial flag data at some special fibers.
When studying logarithmic Higgs fields on parabolic bundles, we will call the whole ensemble a ``parabolic Higgs bundle'',
and we will fix some conjugacy classes of the residues of the logarithmic Higgs field.
The flag data at a special point of $\P^1$ will be related to the eigenspaces of the residues of the field.

In \cite{simpson}, Simpson details similar, closely related moduli problems in which the Higgs field is replaced by a (logarithmic) connection.
One of his computations shows that there are only four cases of numerical invariants which lead to two dimensional moduli spaces.
The first of these is related to the Painlev\'{e} VI equation and is studied in \cite{conn}.
In all four cases, the moduli problem takes a positive integer $r$ as a parameter (the rank of the underlying vector bundle is a fixed multiple of $r$, for one).
The aim of this paper is to study these four families of moduli problems and show that, within each family, the moduli space does not depend on the parameter $r$.
The isomorphisms constructed are not abstract but have modular interpretations given by certain integral transforms on the associated spectral curves.

In Section 2, we define the moduli problem, construct the spectral curve, and refine the notion of spectral curve by showing that the module $M_\phi$
may be lifted to a blowup of the total space of the twisting line bundle.

In Section 3, we use intersection theory on surfaces to study the spectrum $S$ and its refinement and ultimately decide what the Hitchin map is in this context.
The main goal is to show that the possible support subschemes for the module $M_\phi$ do not depend on the parameter $r$.
In other words, for cases with $r>1$, we show that the minimal polynomial of $\phi$ is as small as possible given the numerical constraints,
and is very much lower degree than the characteristic polynomial.
Said yet another way, we prove that the module $M_\phi$ is supported on the reduced subscheme of $S$ and that this scheme is a Gorenstein curve of arithmetic genus one.

Section 4 is a study of the deformation theory of the moduli problem. 
This is done to ensure that the module $M_\phi$ is supported on vanishing scheme of the much smaller minimal polynomial even over arbitrary families parametrized by non-reduced schemes.
Once this calculation has been done, one is assured that the module $M_\phi$ is ``pure of dimension one'' on the reduced spectrum of $S$, 
so in Section 5 we finally apply our Fourier-Mukai transforms to reduce the parameter $r$.
These integral transforms are a means of repeating the result, in families, that Atiyah proved in \cite{atiyah}.
This reduction allows us to prove that the moduli problem is independent of $r$,
and ultimately that the coarse spaces are isomorphic to an open subset of the blowup described in Section 3.

\textbf{Acknowledgements:} The author would like thank Dima Arinkin for posing the question and for countless helpful discussions.
Thanks also to Ed Dewey for his thoughtful comments on the presentation and organization of this paper.
During the course of this research, the author has been partially supported by NSF grants DMS-1452276 and DMS-1502553.

\section{Definitions and Constructions}

The moduli problems of interest in this paper are logarithmic Higgs bundles on $\P^1$ with fixed residue endomorphisms at the points of $\P^1$ where the Higgs field is allowed a pole. 
A logarithmic differential on a smooth curve is a differential 1-form with order one poles at some points.
Let $Z$ be an effective divisor on $\P^1$ consisting of some distinct points $p_i$ all of multiplicity one. 
By $\Omega'$ I mean the line bundle of differentials on $\P^1$ allowed simple poles at the chosen points i.e. $\Omega(Z)$. 

\begin{defn}
  For each $p_i\in Z$, there is a canonical morphism of $\O$-modules $\Res_i:\Omega'\ra \O_{p_i}$. 
  This map is called the residue map (at $p_i$) and it does not depend on a choice of coordinates: 
  to get a local differential of residue 1, pick $\pi$ a local parameter at $p_i$ and take the differential $\frac{d\pi}{\pi}$, sometimes written $\dlog(\pi)$. 
  Its kernel consists of differentials of $\Omega'$ regular at $p_i$. 
  Occasionally I will want to use an extension of this map-- $\Res_i^d:\Sym^d \Omega' \ra \O_{p_i}$; 
\end{defn}
These higher residue maps are multiplicative with respect to each other,
and the kernel of $\Res_i^d$ is the subsheaf of $\Sym^d \Omega'$ consisting of symmetric differentials with poles of order less than $d$ at $p_i$.  
So for any positive exponent $d$ we have canonical short exact sequences:
\[0\ra (\Omega')^{\otimes d}(-Z) \ra (\Omega')^{\otimes d} \ra \bigoplus_i \O|_{p_i}\ra 0 \]
The most important basic result about logarithmic differentials is the following:
\begin{lemma}
  The sum of all the residues of a logarithmic differential is 0. 
  \label{trace}
\end{lemma}
\begin{proof}
  This is true on any smooth projective curve, but we'll prove it here only for $\P^1$.
  Let $x$ be a coordinate for the standard first affine chart of $\P^1$. 
  Any logarithmic differential is represented here as a sum of terms of the form $\frac{c\: d x}{x-a}$. 
  Re-writing this in terms of $y=1/x$ we find:
  \[ \frac{c\: d(x-a)}{x-a} = \frac{c\: d x}{x-a} = \frac{c\: d (1/y)}{(1-ay)/y} = \frac{(-c/y^2)\: d y}{(1-ay)/y} = \frac{1}{(1-ay)} \frac{-c\: d y}{y}\]
  The function $\frac{1}{(1-ay)}$ is 1 at infinity, so the term $\frac{c\: d(x-a)}{x-a}$ contributes $c$ to the residue at $x=a$ and $-c$ to the residue at infinity. 
  Since $\omega$ is a finite sum of such terms, this proves the lemma.
\end{proof}

The residue map can also be extended in an obvious way to vector bundles.  
Let $W$ be a locally free sheaf on $\P^1$, then there is a canonical identification $(\Omega'\otimes W)|_{p_i}\ra W|_{p_i}$. 
This can be obtained by tensoring the usual residue map with the identity map of $W|_{p_i}$. 
In particular, when $W$ is a sheaf of endomorphisms of a vector bundle, this will allow me to talk about the residue of an endomorphism. 

We are now able to define the fundamental moduli functors of interest in this paper. 
In \cite{simpson}, Carlos Simpson computes the possible moduli spaces of parabolic Higgs bundles which can be two dimensional.
There are four cases of the moduli problem we will consider\textemdash
each requires fixing some points of $\P^1$ and then fixing some eigenvalues over at each of the points.
Some of these eigenvalues will enjoy extra multiplicity over the other eigenvalues, which we'll denote by $c_i$.
Unless stated otherwise, $c_i=1$ for every $p_i\in Z$.
In each of the cases, the $\P^1$ vector bundle receiving a Higgs field will always have rank a multiple of some fixed integer $\Theta$.
\begin{enumerate}[I.]
  \item Fix four points $\{p_1,p_2,p_3,p_4\}=Z$ of $\P^1$ and residue eigenvalues $\lambda_{i,1},\lambda_{i,2}$; $\Theta=2$.
  \item Fix three points $\{p_1,p_2,p_3\}=Z$ and three residue eigenvalues at each point; $\Theta=3$.
  \item Fix three points $\{p_1,p_2,p_3\}=Z$ along with two residue eigenvalues at $p_1$, $c_1=2$, and four each at $p_2$ and $p_3$; $\Theta=4$.
  \item Fix three points $\{p_1,p_2,p_3\}=Z$ and two eigenvalues at $p_1$, $c_1=3$, three at $p_2$, $c_2=2$, and six at $p_3$; $\Theta=6$.
\end{enumerate}

With this data fixed, we are able to define the moduli spaces of study in this paper.
\begin{defn}
  Let $T$ be a locally noetherian $\C$-scheme.
  By a rank $r$, degree $d$, \textbf{logarithmic Higgs bundle} parametrized by $T$, 
  we mean a vector bundle $V$ on $\P^1_T$ of rank $\Theta r$
  with a map $\phi:V\ra \Omega_{\P^1_T/T}(Z\times T)\otimes V$ (called the \textbf{Higgs field}) 
  and a trivialization $\iota:\det V \xrightarrow{\sim} \O_{\P^1_T}(d)$ satisfying the following conditions. 
  For $p_i\in Z$ the vector bundle $V_i:=V|_{p_i\times T}$ decomposes as a direct sum $\bigoplus V_{\lambda_{i,j}}$ of bundles
  such that $V_{\lambda_{i,j}}$ is rank $c_i r$ and $\Res_i\phi$ acts on $V_{\lambda_{i,j}}$ as the scalar $\lambda_{i,j}$.
  An isomoprhism of such objects is given by an isomorphism of vector bundles that commutes with the Higgs field map and the trivializations $\iota$. 
  This stack we will denote by $\H_{r,d}^{i}$ (resp. $\H_{r,d}^{ii}$, $\H_{r,d}^{iii}$, or $\H_{r,d}^{iv}$).
  We also require that the bundle $V$ is slope stable, but only with respect to $\phi$-invariant sub-bundles.
  \label{moduliproblem}
\end{defn}

\begin{center}
\[  
\begin{array}{r|l}
  \mbox{Case I} &
  \Res_i(\phi)\sim
  \left(
  \begin{array}{c|c}
    \lambda_{i,1}I_r&0\\
    \hline
    0&\lambda_{i,2}I_r
  \end{array}
  \right)
  \\
  \hline
  \mbox{Case II} &
  \Res_i(\phi)\sim
  \left(
  \begin{array}{c|c|c}
    \lambda_{i,1}I_r&0&0\\
    \hline
    0&\lambda_{i,2}I_r&0\\
    \hline
    0&0&\lambda_{i,3}I_r
  \end{array}
  \right)
  \\
  \hline
  \multirow{2}{*}{Case III}&
  \Res_1(\phi)\sim
  \left(
  \begin{array}{c|c}
    \lambda_{1,1}I_{2r}&0\\
    \hline
    0&\lambda_{1,2}I_{2r}
  \end{array}
  \right)
  \\
  \cline{2-2}
  &
  \Res_i(\phi)\sim
  \left(
  \begin{array}{c|c|c|c}
    \lambda_{i,1}I_{r}&0&0&0\\
    \hline
    0&\lambda_{i,2}I_{r}&0&0\\
    \hline
    0&0&\lambda_{i,3}I_{r}&0\\
    \hline
    0&0&0&\lambda_{i,4}I_{r}
  \end{array}
  \right), i=2,3
  \\
  \hline
  \multirow{3}{*}{Case IV}&
  
  \Res_1(\phi)\sim
  \left(
  \begin{array}{c|c}
    \lambda_{1,1}I_{3r}&0\\
    \hline
    0&\lambda_{1,2}I_{3r}
  \end{array}
  \right)
  \\
  \cline{2-2}
  &
  
  \Res_2(\phi)\sim
  \left(
  \begin{array}{c|c|c}
    \lambda_{2,1}I_{2r}&0&0\\
    \hline
    0&\lambda_{2,2}I_{2r}&0\\
    \hline
    0&0&\lambda_{2,3}I_{2r}
  \end{array}
  \right)
  \\
  \cline{2-2}
  &
  
  \Res_3(\phi)\sim
  \left(
  \begin{array}{c|c|c|c}
    \lambda_{3,1}I_{r}&0&\cdots&0\\
    \hline
    0&\lambda_{3,2}I_{r}&\cdots&0\\
    \hline
    \vdots&\vdots&\ddots&\vdots\\
    \hline
    0&0&\cdots&\lambda_{3,6}I_{r}
  \end{array}
  \right)
  \\
\end{array}
\]
\small Table 1: Required conjugacy classes of the residue endomorphisms.
\end{center}
\pagebreak

When we just want to refer to one of these moduli problems without picking a case, we will write $\H_{r,d}$

\begin{defn}
  There is another version of the moduli problem where one neglects the map $\iota$ and instead only requires that $\det V$ is locally isomorphic to $\O(d)$.
  This less rigid problem we denote by $\H^\circ_{r,d}$ ($\H^{\circ,i}_{r,d}$, etc.).
\end{defn}

In cases I and II, it is easy to see that the only requirement for existence of such a bundle is that the sum of all the eight (resp. nine) eigenvalues is zero.
Proving existence in cases III and IV is not so easy, as we shall see later.

Though the definition of $\H_{r,d}$ gives a direct sum decomposition of the fiber of $V$ at $p_i\in Z$, we will often only want to remember the data of a flag.
To observe the conventions established for parabolic Higgs bundles, let $L_i^\bullet$ be the decreasing filtration of $V|_{p_i}$ given by the existing ordering on the eigenvalues.
Expressly: $L_i^l$ consists of the sum of eigenspaces $V_{\lambda_{i,j}}$ for all but the last $l$ eigenvalues at $p_i$.

\begin{prop}
  The moduli problems $\H$ and $\H^\circ$ are algebraic stacks.
\end{prop}
\begin{proof}
  Rather than directly demonstrating a schematic cover of the stack, consider instead the map to $\Bun_{SL_{\Theta r}^d}\P^1$ - the stack defined as $\H_{r,d}$ but without any Higgs field $\phi$. 
  Now consider a scheme $S$ with a map $(V,\iota)$ to $\Bun_{SL_{\Theta r}^d}\P^1$. 
  We wish to form the fiber product $S\times_{\Bun} \H_{r,d}$; we are hoping for this to be an $S$-scheme.
  With no conditions on $\phi$, the result should be an $S$ scheme which represents $\H^0(-,\END V\otimes \Omega')$.
  That is, for any $f:T\ra S$, the $T$ points should be $H^0(\P^1_T,f^*\END V\otimes \Omega')$.
  
  This is representable by an $S$-scheme in the following way. We want to know $\P^1_T$ global sections of $f^*\END V\otimes \Omega'$.
  By base change, this is the same as $T$ global sections of $R\pi_* \END V\otimes \Omega'$.
  The latter object can be represented as a two term complex $A^0\xrightarrow{d}A^1$ concentrated in degrees zero and one, since we are pushing forward from a relative curve.
  We may even assume that $A^1$ is a vector bundle and $A^0$ is something coherent.
  If $A^0$ isn't flat, there is some point $p$ of $S$ at which it will have a non-trivial first Tor group.
  This would cause the derived restriction of the pushforward complex to have negatively graded cohomology.
  On the other hand, this should just be the cohomology of $\END V\otimes \Omega'$ restricted to $\P^1_{\kappa(p)}$.
  So $A^0$ is a flat coherent sheaf and hence a vector bundle.
  Locally on $S$, $A^0$ and $A^1$ are both free, so the functor we are trying to represent is nothing but the kernel of some matrix $d$
  and this may be represented by some homogeneous linear subset of a trivial $S$ vector bundle.
  
  The conditions on the residues of $\phi$ are all algebraic conditions on this total space,
  and the condition of stability is Zariski open. 
  Since $\H_{r,d}$ has a representable morphism to an algebraic stack, it is also algebraic.
\end{proof}

\subsection{Spectral Curves}

Let $X$ be any scheme and let $L$ be an invertible sheaf on $X$.
We have the standard construction $\pi:|L|\ra X$, where $\pi$ is the affine morphism with structure sheaf $\Sym^\bullet L\check{}$.
The modules on this line bundle have an easy description in terms of $X$ and $L$.
\begin{construction}
  Suppose we have $F$ a quasicoherent sheaf on $X$ equipped with $\phi:F\ra L\otimes V$ an $\O$-linear map. 
  This gives $F$ the structure of a module for the algebra $\Sym^\bullet L\check{}$.
  In other words, $(V,\phi)$ is a module on the total space $|L|$\textemdash we will call it $M_\phi$. 
  
  Conversely, consider a module $M$ on the total space $|L|$.
  This data consists of the $\O_X$ module $\pi_*M$ along with an action map
  \[F\otimes_{\O_X}\Sym^\bullet L\check{}\ra F\]
  This action is multiplicative in the symmetric algebra factor, so it is completely determined by the factor
  \[F\otimes_{\O_X}L\check{}\ra F\]
  but after tensoring by $L$, this is the same as an $L$-twisted endomorphism of $F$.
  \label{totalspace}
\end{construction}

In our case the module $F$ is a vector bundle $V$,
and the Cayley-Hamilton theorem gives a good upper bound on the support of the module $M_\phi$. 
There is a characteristic polynomial
\[P_\phi=y^n - a_1\cdot y^{n-1}+a_2\cdot y^{n-2} \cdots (-1)^n a_n\]
Where $a_1$ is the trace of $\phi$, $a_n$ is the determinant, and in general $a_i$ is a global section of $L^i$. 
\begin{defn}
  If we let the heretofore formal variable $y$ stand for the tautological section of $\pi^*L$ on $|L|$, then $P_\phi$ is a global section of $\pi^*L^n$ on $|L|$. 
  The vanishing scheme of $P_\phi$ we will denote $S_\phi$ or $S$, for it is the spectrum of the operator $\phi$. 
  Note that the spectrum is a finite cover of the base $X$.
  \label{spectrum}
\end{defn}

\begin{lemma}
  The module $M_\phi$ is supported on the spectral cover $S_\phi$. 
  \label{cayleyhamilton}
\end{lemma}
\begin{proof}
  The construction and claim are both local on the base $X$, so I can assume that $X = \Spec A$ and that $s$ is a trivializing global section of $\Omega'$. 
  Then $\frac{\phi}{s}$ is an endomorphism of the vector bundle $V$. 
  The characteristic polynomial of $\frac{\phi}{s}$ is easy to deduce from the characteristic polynomial of $\phi$: its weight $i$ invariant is $\frac{a_i}{s^i}$. 
  Hence the operator 
  \[\sum_{i=0}^n (-1)^i \left(\frac{\phi}{s}\right)^{n-i}\frac{a_i}{s^i}\] 
  annihilates the vector bundle $V$ by the Cayley-Hamilton theorem. But this operator can be factored: 
  \[s^{-n} \sum_{i=0}^n (-1)^i(\phi)^{n-i} a_i=s^{-n} P_\phi\] 
\end{proof}

Note that both the equivalence $(V,\phi)\Leftrightarrow M_\phi$ and the construction of the characteristic polynomial respect pullbacks.

This construction prompts the definition of a moduli problem closely related to $\H_{r,d}^i$ and $\H_{r,d}^{ii}$.
Fix a test scheme $T$, we will work with $\Omega'$ twisted endomorphisms on $\P^1_T$.
For case I, let $y^2 - a_1 y + a_2$ be a characteristic polynomial whose value in the fiber at $p_i$ 
yields the characteristic polynomial of an operator with eigenvalues $\lambda_{i,1}$ and $\lambda_{i,2}$.
There is only one choice for $a_1$, but the coefficient $a_2$ moves in a one parameter family.

Likewise for case II, let $y^3-a_1 y^2+a_2 y-a_3$ be picked so that in its $p_i$ fiber it has roots $\lambda_{i,1},\lambda_{i,2},\lambda_{i,3}$.
Only $a_3$ has a one parameter family of options while the other two coefficients are determined.

\begin{defn}
  For any choice of characteristic polynomial $f$ as described above, we get a corresponding spectral curve $S_f$ over $\P^1$ which is a double (resp. triple) cover.
  On this curve $E_f$, pick $M$ a rank $r$, degree $2r+d$ (resp. $3r+d$) sheaf which is pure of dimension one on fibers over points of $T$, and which is flat over $\P^1_T$.
  Let $\pi:S_f\ra \P^1$ be the covering map and give an isomorphism $\iota:\det \pi_*M \ra \O(d)$. 
  Isomorphisms are taken to be isomorphisms of $|\Omega'|$ modules and are required to commute with $\iota$.
  Isomorphisms can only exist between two bundles on the same $S_f$.
  This defines a stack we will call $\E_{r,d}^{i}$ (resp. $\E_{r,d}^{ii}$).
  \label{substackdefn}
\end{defn}
Because of the previous construction, this stack has a natural map to the stack $\E_{r,d}^{i}$ (resp. $\E_{r,d}^{ii}$). 
\begin{lemma}
  This map of stacks is fully faithful and the characteristic polynomial associated to a module on $S_f$ is $f^r$.
  \label{substack}
\end{lemma}
\begin{proof}
  First we address the calculation of the characteristic polynomial of $(V,\phi) = \pi_*M$.
  The data of the characteristic polynomial consists of some sections of line bundles on $\P^1_T$.
  It is therefore enough to verify that the claimed coefficients are correct on a neighborhood of each of the associated points of $P^1_T$.
  These points are in bijection with the associated points of $T$.
  Let $t\in T$ be an associated point and let $\tau$ be $\Spec$ of the stalk at $t$.
  Over $\P^1_\tau$, the cover $S_f$ is generically smooth and so the bundle $M$ is generically a vector bundle.
  In fact, generically on $\P^1_\tau$ we can assume that $M$ is actually a trivial vector bundle on $S_f$.
  Regarded as a Higgs bundle, the structure sheaf of $S_f$ has characteristic polynomial $f$,
  so the characteristic polynomial of a direct sum of $r$ copies of $\O_{S_f}$ is $f^r$.
  As the associated points of $\P^1_T$ are all generic points of $\P^1_t$ for $t$ an associated point of $T$, we are done.

  We have already seen that the isomorphisms between two Higgs bundles are the same as the isomorphisms between their corresponding modules on $|\Omega'|$.
  This is the same as giving an isomorphism between two vector bundles on a fixed $S_f$, so the described functor is faithful and full.
\end{proof}

In cases III and IV, one can still define stacks $\E_{r,d}^{iii}$ and $\E_{r,d}^{iv}$, but it is not enough to consider a spectral curve $S_f$ in $|\Omega'|$.
For instance, consider the setup for case III and take a spectral curve $S$ in $|\Omega'|$ 
and which passes through the ten specified points in the three designated fibers.
If we consider the Higgs bundle corresponding to the structure sheaf of this spectral curve, the residue endomoprhism at $p_1$ will not be diagonalizable,
but rather will consist of two Jordan blocks of size two with eigenvalues $\lambda_{1,i}$.
For this reason we must consider extra structure coming from our diagonalizable residue endomorphisms.

\subsection{Blowups}
We have seen that a vector bundle with a twisted endomorphism on $X$ can be viewed instead as a module on the total space of the twisting line bundle $L$,
but in the case of $\H_{r,d}$ there is even more structure.
First we need to set up some new (relative) surfaces over $T$. 
For $p_i\in Z$ and $\lambda_{i,j}$ one of the designated eigenvalues at $p_i$, let $e_{i,j}$ be the point of $|\Omega'|$ in the fiber over $p_i$ with residue $\lambda_{i,j}$.
Let $\bar{\sigma}:\bar{B}\ra|\Omega'|$ be the blowup of $|\Omega'|$ at all of the $e_{i,j}$.
We will write $E_{i,j}$ for the exceptional divisor over $e_{i,j}$.
$\bar{B}$ has an easy to describe open subset consisting of the compliment of the strict transform the fibers in $|\Omega'|$ over points of $Z$.
This open subset of $\bar{B}$ call $\sigma:B\ra |\Omega'|$.
Note that these constructions are constant in $T$:
we may first construct $B\hookrightarrow \bar{B}\ra \P^1_{\C}$ and then pull this back from $\Spec \C$ to $T$.

\begin{lemma}
  Let $(V,\phi)$ be a log Higgs bundle on $\P^1_T$ with eigenvalue $\lambda$ of multiplicity $s$ at $p\times T$. 
  Further assume that the residue endomorphism at $p$ has trivial $\lambda$ Jordan blocks, i.e. the $\lambda$ part of the residue endomorphism is scalar. 
  Then the characteristic polynomial $P_\phi$ has multiplicity $s$ at the point $(p,\lambda)$ of $|\Omega'|$.
  \label{multiplicity}
\end{lemma}
\begin{proof}
  By subtracting a scalar endomorphism from $\phi$, we may assume that $\lambda = 0$.
  Then the ideal defining $(p,\lambda)$ is the one generated by $x$ and $y$, where $x$ is a function on $\P^1$ vanishing at $p$ and $y$ is a trivialization of $\Omega'$.
  The characteristic polynomial, recall, is of the form:
  \[P_\phi=\sum_{i=0}^{n} (-1)^iy^{n-i}p_i\]
  We wish to show that $P_\phi\in (x,y)^s$.
  The higher terms of $P_\phi$ are in $(x,y)^s$, so it remains to show that $p_{2r-s+j}\in (x,y)^j$ for $0<j\leq s$.
  The invariant $p_{2r-s+j}$ can be calculated as the trace of $\wedge^{2r-s+j}\phi$ acting on $\bigwedge^{2r-s+j} V$.
  Now restrict $(V,\phi)$ to the closed subscheme defined by $x^s$.
  Here the module $M_\phi$ splits as $M_\lambda \oplus M_\sim$ where $M_\lambda$ is the part supported on $(p,\lambda)$ and $M_\sim$ is everything else.
  There is a corresponding decomposition $(V,\phi) = (V_\lambda,\phi_\lambda)\oplus (V_\sim,\phi_\sim)$ in which $V_\lambda$ is a rank $s$ bundle and $V_\sim$ is rank $2r-s$.
  The operator $\phi$ is divisible by $x$ when restricted to $V_\lambda$.
  On wedges, this direct sum behaves like:
  \[\bigwedge^{2r-s+j}V = \bigwedge^s V_\lambda\otimes \bigwedge^{2r-2s+j}V_\sim \oplus\cdots \bigwedge^j V_\lambda \otimes \bigwedge^{2r-s}V_\sim\]
  The trace of $\phi$ acting on the term $\bigwedge^{j+l} V_\lambda \otimes \bigwedge^{2r-s-l}V_\sim$
  is the product of trace of $\phi_\lambda$ acting on $\bigwedge^{j+l} V_\lambda$ and the trace of $\phi_\sim$ acting on $\bigwedge^{2r-s-l}V_\sim$.
  Since the operator $\phi_\lambda$ is divisible by $x$, the trace on $\bigwedge^{j+l} V_\lambda$ is divisible by $x^{j+l}$.
  Hence the overall trace of $\phi$ acting on $\bigwedge^{2r-s+j}V$ is divisible by $x^j$.

  To demonstrate that $P_\phi$ has multiplicity no greater than $s$ at $(p,\lambda)$ we only need consider $(V,\phi)$ restricted to the fiber over $p$.
  Here $P_\phi$ is the spectrum of an operator with $\lambda$ as an eigenvalue of multiplicity exactly $s$, so $P_\phi$ cannot have multiplicity greater than $s$.
\end{proof}

\begin{cor}
  Let $(V,\phi)$ be a log Higgs bundle over a field $k$ and suppose the residue endomorphism at $p$ is diagonalizable.
  Denote by $\sigma:\bar{B}\ra |\Omega'|$ the blowup of $|\Omega'|$ at the eigenvalues of the residue endomorphism in the fiber over $p$.
  If $S'_\phi$ is the strict transform of $S_\phi$ to $\bar{B}$ and $F'$ is the strict transform of $F$, the fiber over $p$,
  then $S'$ does not meet $F'$.
  \label{stricttransform}
\end{cor}
\begin{proof}
  Let $C\sub k$ be the set of eigenvalues of the residue endomorphism, $s_\lambda$ being the multiplicity of $\lambda$ as an eigenvalue for $\lambda\in C$.
  By the previous lemma, the multiplicity of $S_\phi$ at $(p,\lambda)$ is $s_\lambda$.
  These multiplicities may be used to compute that 
  \[\sigma^* S = S' + \sum_{\lambda\in C} s_\lambda E_\lambda\]
  Since $S.F = \Theta r$, $\sigma^* S .\sigma^* F = \Theta r$ and we calculate:
  \begin{align*}
  \Theta r &=(S' + \sum_{\lambda\in C} s_\lambda E_\lambda).(F' + \sum_{\lambda\in C} E_\lambda)\\
  &= (S' + \sum_{\lambda\in C} s_\lambda E_\lambda).F'\\
  &=S'.F' +\sum_{\lambda\in C} s_\lambda E_\lambda.F'\\
  &=S'.F' + \Theta r
  \end{align*}
  And so we conclude $S'.F'=0$.
\end{proof}

Assume now that $(V,\phi)$ is a $T$ family from $\H_{r,d}$. Because of the above corollary, the strict transform of $S_\phi$ to $\bar{B}$ is contained in $B$.
The open subset $B$ is affine over $|B|$ so it is easy to give $M_\phi$ the structure of a module on $\bar{B}$.

\begin{prop}
  As before, let $(V,\phi)$ be a log Higgs bundle on $\P^1_T$ with scalar $\lambda$ part of its residue endomorphism at $p$, of dimension $s$.
  Locally on $|\Omega'|$, the module $M_\phi$ may be considered a module on the blowup of $|\Omega'|$ at the point $(p,\lambda)$; we will call this new module $N_\phi$.
  In a neighborhood of the point $(p,\lambda)$, $N_\phi$ is supported on the open subset of the blowup away from the strict transform of the fiber of $|\Omega'|$ over $p$.
  In this neighborhood, the module $N_\phi$ is supported on the strict transform of $S_\phi$.

  The converse is also true: 
  Assume that $S'\sub B$ is a Cartier divisor finite over $\P^1_T$ and that $M$ is a coherent $S'$ module which is a vector bundle over $\P^1$.
  Then the corresponding Higgs bundle on $\P^1_T$ has diagonalizable fiber endomorphisms at $p$.
  \label{stricttransformmodule}
\end{prop}

\begin{proof}
  The open subset $B$ of the full blowup at $(p,\lambda)$ is affine over $|\Omega'|$ and is generated by a function $\tau$ which satisfies $x\tau = y-\lambda$.
  Here we assume we have trivialized $\Omega'$ by $y$ in such a way that $x$ and $y-\lambda$ define $(p,\lambda)$.
  Pick a principal open subset $U$ of $|\Omega'|$ containing $(p,\lambda)$ and such that $S_\phi$ only intersects the fiber over $p$ only at $\lambda$.
  The restriction of $M_\phi|_U$ to the fiber $\pi\i(p)$ is a module supported on the subscheme $(p,\lambda)$ only and not a nilpotent neighborhood,
  so the operator $\phi-\lambda\id$ sends $M_\phi|_U$ into $x M_\phi|_U$.
  The module $M_\phi$ is torsion free over $\P^1_T$, so multiplication by $x$ is an injection from $M_\phi$ into itself.
  Since localization is an exact functor, the module $M_\phi|_U$ is also torsion free for $x$,
  so there is a unique endomorphism of $M_\phi$ deserving to be called $\frac{y-\lambda}{x}$.
  This shows that $M_\phi$ may be considered a module on this open subscheme of the blowup; this module will be called $N_\phi$.

  The strict transform of $S_\phi$ may be calculated directly from its defining equation $P_\phi$.
  It is convenient again to assume that our eigenvalue $\lambda = 0$ by a translation of the bundle $|\Omega'|$.
  After this translation and by taking an open subset $U$ as before, we can write
  \[P_\phi = \sum_{i=0}^n (-1)^i y^{n-i} p_i(x) = \sum_{i=0}{n} (-1)^i (\frac{y}{x})^{n-i}x^{n-i} p_i(x)\]
  The strict transform of $S_\phi$ is the vanishing of the function $P_\phi$ with as many factors of $x$ removed as possible.
  Call this result $Q_\phi$, so $x^e Q_\phi = P_\phi$ and $x\nmid Q_\phi$.
  We know that $P_\phi$ annihilates $M_\phi|_U$, and that $x$ is an injection on $M_\phi$, so $Q_\phi$ must also annihilate $M_\phi|_U$.

  To see the converse, consider the fact that, near the point $(p,\lambda)$, the endomorphism $\phi-\lambda$ of $M$ may be written as a multiple of $x$.
  This shows that the endomorphism $\phi$ induced on $V|_p = \pi_*\sigma_*N_|p$ is diagonalizable.
  Within this fiber, the surface $B$ splits into a few components: one exceptional divisor for each eigenvalue at $p$.
  This implies that the $T$-vector bundle $V|_p$ splits up as a sum $\bigoplus V_\lambda$, which are each vector bundles themselves as summands of a vector bundle.
\end{proof}

Return to the case when $(V,\phi)$ is a $T$ family from $\H_{r,d}$.
For any given point $p\in Z$, we can pick a neighborhood $U\ni p$ of $\P^1_T$ and cover the fiber $\pi\i U$ with open sets of the form required by the previous lemma.
This lets us define $N_\phi$ as a module on $B$. 
The strict transform of the spectrum, $S'_\phi$, supports the module $N_\phi$.
Since $S'_\phi$ is a closed subset of $\bar{B}$ not intersecting the strict transforms of the fibers over the points of $Z$,
we may regard $N_\phi$ to be a module on $B$ or on $\bar{B}$.

\section{The Class of the Spectral Divisor}

Let $k$ be a field containing $\C$ and let $(V,\phi)$ be a point of $\H_{r,d}$. We will be able to use some intersection theory to study the behavior of $N_\phi$ on $\bar{B}$.
We first must build a projective surface over $\P^1_k$. 
The total space of $\Omega'$ is an open subset of $\P(\O \oplus \Omega')$ and so we blow up the points $e_{i,j}$ in this latter surface and call the result $W$.
Note that $\bar{B}$ is an open subset of $W$ but $S'$, the strict transform of the spectrum, is still closed in $W$ and hence $N_\phi$ is a coherent module on $W$.
The divisor class group of $W$ is the free abelian group generated by
\begin{itemize}
  \item $F$ a fiber of the map to $\P^1_k$
  \item $\infty$ the `section at infinity' $\P(\Omega')\sub \P(\O\oplus \Omega')$
  \item $E_{i,j}$ the exceptional loci of the blowup
\end{itemize}
We also have the zero section of $|\Omega'|$, which is $[0]=\infty +2F$ in case I and $[0]=\infty + F$ in cases II, III, and IV.
Basic calculations show that $F.\infty=F.0$, $F^2=0$,
and $[0]^2=-\infty^2 = 2$ in case I and $1$ in cases II, III, and IV.
The $E_{i,j}$ are orthogonal to everything but themselves, and of course the square to $-1$.

We have designated $S'\sub W$ as the strict transform of the spectral curve $S$. 
Both $S$ and $S'$ are presented as Cartier divisors in $W$ and  we wish to determine the class of $S'$.
The pairing $S'.\infty = 0$ since $S\sub \bar{B}$, and $S'.F = \Theta r$ since $S'$ restricted to a fiber is the spectrum of a $\Theta r$ dimensional operator.
Because of Corollary \ref{stricttransform} we know that $S'$ should pair with the exceptional divisor at $(p_i, \lambda_{i,j})$ to $c_{\lambda_{i,j}}r$.
We now know enough about $S'$ to verify:
\[ [S]=\Theta r [0] - \sum_{p_i\in Z} \sum_j c_{\lambda_{i,j}}r E_{i,j}\]

This divisor is $r$ times a divisor we call $\Delta$, i.e. equal to the divisor $S'$ in the case $r=1$.
Since the divisor $S$ does not depend on the Higgs field at all or on the parameter $d$, 
we study the (complete) linear series of $\Delta$ and its multiples to understand which spectral curves may arise.
Before studying this linear series, it is convenient to complete the definition of $\E_{r,d}$:
\begin{defn}
  Let $T$ be a test scheme and let $E\sub W\times T$ be an effective Cartier divisor in the class $\Delta$.
  Assume further that $E$ is contained in the open subset $B\times T$.
  Let $M$ be a coherent sheaf on $E$, which is flat of rank $\Theta r$ over $\P^1_T$, and which is pure of dimension of one on fibers over points of $T$.
  Further require that $M$ is length $c_{i,j}r$ when restricted to $E_{i,j}$.
  We require that $V=\pi_* M$ is a degree $d$ vector bundle on $\P^1$.
  Giving an isomorphism $\iota$ from the determinant of $V$ to the line bundle $\O(d)$ 
  upgrades an object of $\E^\circ$ to an object of $\E$, as before.
  Isomorphisms are just isomorphisms of coherent sheaves on a fixed Cartier divisor $E$ which commute with the map $\iota$ if it is given.
  \label{substackdefn2}
\end{defn}

This is now a complete description of the stack $\E$ in all cases.
The same argument as in the proof of Lemma \ref{substack} shows that the characteristic polynomial of the Higgs bundle is the $r$th power of the defining equation for $\sigma_* E$.
In cases I and II it coincides with the description given in Definition \ref{substackdefn}.
Because of Lemma $\ref{stricttransformmodule}$, the corresponding Higgs bundle on $P^1_T$ will satisfy the criteria to be an object of $\H$.

\begin{lemma}
  The divisor $\Delta$ can be represented by either a divisor at infinity or by the strict transform of a spectrum, which is contained in $B$.
  In any case, the corresponding subscheme of $W$ is connected.
  The subscheme has arithmetic genus one and only Gorenstein singularities.
  \label{arithmeticgenus}
\end{lemma}
\begin{proof}
  The divisors at infinity are connected, by inspection, in all the four cases we consider.
  In the case of a divisor on $B$ representing $\delta$, we write $S'$ for the subscheme of $W$ and $S$ for the subscheme of the pushforward down in $\P(\Omega'\oplus \O)$.
  We begin by considering the spectral curve $S\sub |\Omega'|$ in the case $r=1$.
  This is the zero set of a degree $\Theta$ relation from $\pi^*\Omega'$ on $|\Omega'|$, and so its structure sheaf is
  \[\pi_* \O_S \isom \O\oplus (\Omega')\i \oplus\cdots\oplus (\Omega')^{1-\Theta}\]
  as a bundle on $\P^1$.
  In all cases, then, $h^0(\O_S)=1$ and $h^1(\O_S)=1,1,3,10$ in cases I, II, III, and IV respectively.

  In cases I and II, we see that $S'$ is connected simply because $S\iso S'$ and $S$ is connected by the above cohomological calculation.
  In cases III and IV, the spectrum $S$ is connected by the same argument,
  but since the blowup changes the abstract curve in these cases, we need to work a little harder to show that $S'$ is still connected.

  Consider that in case III, the divisor $S$ is $4[0]$ in the divisor class of the Hirzebruch surface.
  If $S'$ is disconnected, it is because $S$ has irreducible components meeting at the centers of the blowup $\sigma$.
  As $S$ is the vanishing of a monic polynomial, it can only be factored of the form $k[0] + (4-k)[0]$.
  In any case, the two components will intersect in at least 3 points counted with multiplicity.
  However, the only points of the blowup where the curve will actually change are the centers for $E_{1,1}$ and $E_{1,2}$.
  Take $e_{1,1}$- the components multiplicities at this point must add to two, so at most their intersection here is 1.
  Likewise at $e_{1,2}$, so there must be one other intersection point of the two components, which will not be resolved in the blowup $\sigma$.

  Now consider case IV. The divisor $S$ is $6[0]$. 
  If $S$ factors as $[0] + 5[0]$, then the components have 5 points of intersection in the Hirzebruch surface.
  The smaller piece can intersect the larger with multiplicity $2$ at one of the $e_{1,i}$ and multiplicity $1$ at one of the $e_{2,i}$. 
  At any rate, there are still more intersections.
  If $S$ factors as $2[0]+4[0]$ or $3[0]+3[0]$, then the intersections are at least 8,
  but the maximum intersections at the blowups are 2 for the $e_{1,i}$ and 1 for the $e_{2,i}$, so there is still at least one point of intersection after the strict transform.

  Moreover, the spectral curve and its strict transform are reduced: The curve $S$ in $|\Omega'|$ is generically reduced since it has multiplicity 1 at some of the points $e_{i,j}$.
  There can be no embedded points for either $S$ or $S'$ as they are divisors in their respective smooth surfaces.
  Hence the connectedness of $S'$ implies that $h^0(\O_{S'})=1$.

  To see that $S'$ has arithmetic genus one in cases I and II, only consider that $S\iso S'$ 
  and that the structure sheaf of $S$ is $\O\oplus \O(-2)$ in case I and $\O\oplus \O(-1)\oplus \O(-2)$ in case II.
  These vector bundles on $\P^1$ both have $h^1=1$. 
  Now consider case III where $\O_S\isom \O\oplus\O(-1)\oplus\O(-2)\oplus\O(-3)$ as $\O_{\P^1}$ modules.
  Then $h^1(\O_S)=3$ plainly. Since the curve $S$ has multiplicity 2 at each of $e_{1,1}$ and $e_{1,2}$, it loses two from its arithmetic genus in the blowup.
  Likewise, in case IV the structure sheaf of $S$ is $\O\oplus\cdots\oplus\O(-5)$ with $h^1=10$.
  At points $e_{1,1}$ and $e_{1,2}$, $S$ is multiplicity 3 and so loses 3 arithmetic genus from each of these points in the blowup.
  At the three points $e_{2,j}$, $S$ is multiplicity 2 and so loses 1 arithmetic genus at each of these points.
  
  Now we have that $S'$ is a reduced Cartier divisor in $W$ with $h^0(\O_{S'})=h^1(\O_{S'})=1$.
  These curves have at worst Gorenstein singularities simply by virtue of being Cartier divisors in the smooth surface $W$.
\end{proof}

Our analysis of $S'$ so far has been dependent on knowing that $S'$ comes to us as the strict transform of the spectrum of an operator.
To get further, we need to show that these spectra exist and then study the complete linear series associated to $\Delta$ and its multiples.

\begin{lemma}
  The divisor at infinity $\Delta_\infty$ has trivial normal bundle iff
  \[\sum_{i,j} c_{i,j} \lambda_{i,j} = 0\]
  and can be deformed away from itself entirely to lie in the open subset $B$.
  In this case, we are guaranteed existence of points in $\E_{r,-1}$.\footnote{Later we will use this to prove existence of stable modules for other choices of $d$.}
  \label{existence}
\end{lemma}
\begin{proof}
  We recall that the divisor at infinity is 
  \[\Delta_\infty = \Theta \infty + \sum_{p_i\in Z} c_i F_i\]
  where the divisor $F_i$ is the strict transform to $W$ of the fiber over $p_i$ in the Hirzebruch surface.
  Note that this is not simply a divisor class, but a specific divisor.
  The sheaf which controls deformations of $\Delta_\infty$ is the restriction of the line bundle $\O(\Delta)$ to $\Delta_\infty$, which we will simply call the normal bundle
  We begin by trying to establish existence of a global section of $\O(\Delta)/\O$.
  Breaking the subscheme $\Delta_\infty$ into its irreducible components, we get the exact sequence
  \[0\ra \O(\Delta)|_{\Di}\ra \O(\Delta)|_{\Theta \infty}\oplus \bigoplus_{p_i\in Z} \O(\Delta)|_{c_i F_i} \ra \O(\Delta)|_\mathcal{I}\ra 0\]
  The subscheme $\mathcal{I}$ is the finite one given by the intersection of $\Theta$ times the infinity divisor and the fibers $c_i F_i$.
  The normal bundle will have a nonzero global section iff the second map is not injective.
  In every case I through IV, the latter two sheaves have spaces of global sections with an equal number of global sections.
  Explicitly, these spaces are 8, 9, 16, and 36 dimensional in cases I, II, III, and IV, respectively.

  So existence of a global section of the normal bundle comes down the vanishing of the determinant of a single matrix.
  This matrix is a function of the eigenvalues $\lambda_{i,j}$, 
  because these determine how global sections of $\O(\Delta)$ on $c_i F_i$ look as sections at the intersection points $\mathcal{I}$.
  Suppose that this determinant vanishes and the normal bundle has a global section $s$.
  Assume that $s$ vanishes somewhere on $\Di$.
  If $s$ only vanishes at some closed points of one of the reduced divisors $\infty$ or $F_i$, then it implies that the normal bundle there is of positive degree,
  as $s$ produces an injection from the structure sheaf of that reduced divisor into the normal bundle.
  Hence $s$ must either vanish entirely or never vanish, on each of the reduced divisors.
  Since $s$ vanishes somewhere and $\Di$ is connected, $s$ vanishes on the reduced subscheme of $\Di$.
  Now $s$ has been downgraded from a section of $\O(\Delta)/\O$ to a section of $\O(\Delta - \infty - \sum F_i)/\O$.

  One can inductively show from here that $s$ must vanish on the whole subscheme as follows:
  Give $\Delta_\infty -\infty - \sum F_i$ the name $\Delta_1$.
  From here we give a decreasing sequence of sub-divisors $\Delta_i$, such that $\Delta_i - \Delta_{i+1}$ is one of the reduced divisors $\infty$ or $F_i$.
  By picking this sequence correctly, we find that $\O(\Delta_i)$ is a negative degree line bundle on the divisor $\Delta_i - \Delta_{i+1}$.
  Hence the section $s$, when it is a section of $\O(\Delta_i)/\O$ is consequentially a section in the subsheaf $\O(\Delta_{i+1})/\O$.
  The section $s$ is then, inductively, a section of the zero sheaf once $\Delta_i$ is the empty divisor.
  Listed below are these sequences for each of the four cases:

  \begin{center}
  \begin{tabular}{|c|c|c|c|}
    \hline
    I&II&III&IV\\
    \hline
    $\infty$&$2\infty$ & $3\infty + F_1$& $5\infty + 2 F_1 + F_2$\\
    $0$ & $\infty$& $2\infty + F_1$& $4\infty + 2 F_1 + F_2$\\
    & $0$& $\infty + F_1$& $3\infty + 2 F_1 + F_2$\\
    & & $\infty$& $3\infty + F_1 + F_2$\\
    & & $0$& $2\infty + F_1 + F_2$\\
    & & & $2\infty + F_1$\\
    & & &$\infty + F_1$ \\
    & & &$\infty$ \\
    & & &$0$ \\
    \hline
  \end{tabular}
  \end{center}

  What we have shown is that a single determinant in terms of the $\lambda_{i,j}$ vanishes iff the normal bundle to $\Delta_\infty$ is a trivial line bundle.
  If the normal bundle is trivial, take $s\in H^0(\Delta_\infty, \O(\Delta))$ to be a nowhere vanishing section, 
  and lift this to a global section of $\O(\Delta)$ on $W$.
  This may be done because the first cohomology of $\O$ vanishes for the surface $W$.
  Now $V(s)$ is another divisor in the class $\Delta$ which is contained in $B$ since the section $s$ vanishes nowhere on $\Delta_\infty$.
  Call this divisor $S$.

  The canonical class of the surface $W$ is \[\omega_W = -2\infty - \sum_{p_i\in Z} F_i\] and the normal bundle of $S$ is the pullback of $\Delta$.
  Both of these divisors can be represented by divisors on the compliment of $B$ in $W$, so the normal bundle
  and, by the adjunction formula, the canonical bundle of $S$ are both trivial.
  Note that, though $S$ may not be smooth, it gets a canonical line bundle because it has at worst Gorenstein singularities.

  The curve $S$ is reduced- every of its possible components must pass through the fiber over $p_3$, and $S$ intersects each of $E_{3,j}$ with multiplicity one.
  None of the components of $S$ could be the divisor $E_{3,j}$ itself, so each component of $S$ must be generically reduced 
  and there can be no embedded points on a Cartier divisor in $W$.

  The pushforward of $S$ to $\P(\Omega'\oplus \O)$ is in the class of a spectral curve of the type in $\H_{r,d}$.
  Because it doesn't intersect infinity, this pushforward is contained in $|\Omega'|$ and is, in fact, the vanishing of a characteristic polynomial.
  The cohomological calculations of \ref{arithmeticgenus} show that $S$, as the strict transform of its pushforward, must be connected.
  By \cite[Cor. 1.25]{spanish}, there will exist a generalized vector bundle on $S$ of the type required in $\E_{r,d}$.
  Taking $(V,\phi)$ to be the corresponding Higgs bundle to one of these points, say, of $\E_{1,1}$, we compute the trace of $\phi$.
  The trace of $\phi$ is a global section of $|\Omega'|$ which has residue $\sum_j c_i \lambda_{i,j}$ at $p_i$.
  The only way such a section can exist is if the sum 
  \[\sum_{p_i\in Z}c_i \sum_j \lambda_{i,j} = 0\]
  Since the existence of the bundle was guaranteed based solely on the vanishing of a single determinant in terms of the $\lambda_{i,j}$,
  that determinant condition must either be equivalent to the vanishing of this sum, or be never satisfied, for any choices of $\lambda_{i,j}$.
  The later case can be ruled out by a careful choice of coefficients of a characteristic polynomial in cases I and II,
  and by building $S$ as a union of rational double covers of $\P^1$ in cases III and IV.
  Hence, the linear condition on the sum of the eigenvalues, which we knew was a necessary condition for existence of Higgs bundle,
  turns out to be sufficient as well.
\end{proof}

From now on, it is a standing assumption that this sum of eigenvalues, added with multiplicity, is zero.
This guarantees existence of appropriate spectral curves in $B$ which support the Higgs bundles we study in $\H_{r,d}$.

\begin{lemma}
  The divisor $\Delta$ has self intersection zero, is base point free, and determines a map to $\P^1_k$.
  The divisor $S$ is a pullback of a divisor in the class of $r$ times a point on $\P^1$, and as such is equal to a sum of divisors in the class $\Delta$.
  \label{hitchinmap}
\end{lemma}
\begin{proof}
  First note that 
  \[\Delta^2 = \Theta^2 [0]^2 + \sum_{p_i\in Z}\sum_j c_{\lambda_{i,j}}^2 (E_{i,j})^2\]
  In every case, this amounts to zero.

  We would now like to demonstrate that this divisor is base point free and determines a map to $\P^1$.
  Points of $W\setminus B$ are not in the base locus because all of these spectral curves are contained in $B$.
  On the other hand, we may also represent $\Delta$ by a divisor contained in $W\setminus B$.
  Specifically,  the divisor $\Di$ studied above in the proof of Lemma \ref{existence}.
  Thus the divisor $\Delta$ defines a regular map $h:W\ra \P^n_k$.

  So far we know all our spectral curves (case $r=1$) and the divisor at infinity are effective divisors in the class $\Delta$.
  We would like to show that these divisors we've named are the only in the class $\Delta$- in other words that $h^0(\O(\Delta)) = 2$.
  Since there are at least two effective representatives of $\Delta$, $h^0(\Delta)\geq 2$.
  If $h^0(\Delta)>2$, we have two cases.
  First suppose that the image of $h$ is one dimensional.
  In every case, one of exceptional divisors $E$ satisfies $E.\Delta=1$.
  Hence the map $h$ is degree one when restricted to $E$.
  Now the image of $h$ is dimension one, irreducible, and contains a degree one embedding of a $\P^1$.
  The only possibility is that the image of $h$ is a line, which contradicts the assumption $h^0(\Delta)>0$.
  Now suppose that the image of $h$ is two dimensional.
  Then every codimension two linear subspace of $\P^n$ intersects the image of $h$, so any two divisors in $\Delta$ will intersect.
  This contradicts the fact that $\Delta^2=0$.
  Hence we must have $h^0(\Delta)=2$ and $h:W\ra\P^1_k$.
  Arrange that $h\i(\infty)$ is the divisor at infinity, so $B = h\i(\A^1)$.
  
  The spectral curve $S$ is in the class $r\Delta$, so it represented by a global section of the pullback $h^*\O(r)$.
  We claim that the map $h$ is $\O$-connected.
  Since $h$ is dominant to $\P^1$ and since $W$ has only one associated point, the map $h$ is flat.
  The pushforward $h_*\O_W$ is torsion free, hence a vector bundle.
  We know that, fiberwise, the rank of $h_* \O_W$ is $1$, so by \cite[Corollary III.9.4]{hartshorne} $h_* \O_W$ must be a line bundle.
  Since $h_* \O_W$ has the nowhere vanishing global section $1$, $\O_{\P^1}\isom h_* \O_W$ and $h$ is $\O$-connected.

  This implies surjectivity of pullback on global sections 
  \[H^0(\P^1,\O(r))\ra H^0(W, \O(r\Delta))\]
  Hence $S$ must be the pullback of some degree $k$ divisor, call it $D$, on $\P^1_k$.
  If $D$ is supported at multiple points of $\P^1_k$, then the module $N_\phi$ would split as a direct sum, and so then would the Higgs bundle $(V,\phi)$.
  This is impossible since we assume $(V,\phi)$ is stable.
  So the spectral curve and the characteristic polynomial are both $r$-th powers.
\end{proof}

This characterization of the spectrum $S'$ allows us to prove the main result of this section.

\begin{lemma}
  Any object $(V,\phi)\in \H_{r,d}(k)$ is isomorphic to an object in the image of $\E_{r,d}(k)$.
  \label{victoryforfields}
\end{lemma}
\begin{proof}
  We would like to show that $M_\phi$ is actually supported on the reduced subscheme of the spectral curve.
  The spectral curve $S$ is of the form $V(f^r)$ where $f$ is some function on $B$. 
  This is because the compliment of $B$ is in $W$ is the support of the divisor at infinity of $h$.
  Now $f$ is an endomorphism of $N_\phi$, which in turn gives an endomorphism of the pair $(V,\phi)$.
  Since $(V,\phi)$ is stable, it's endomorphisms form a division algebra, so $f$ is necessarily zero since its $r$ power is zero.
\end{proof}

\section{Deformation Theory of $\H$}

In order to prove the isomorphism of moduli spaces desired, we must treat the case of families parametrized by a non-reduced base.  
Since stability is only a condition on fibers above closed points, this requires us to check a certain amount of deformation theory. 
To be concrete, suppose that $R$ is an Artin local ring, and call its residue field $k$. Suppose that $(V,\phi)$ is a point of $\H_{r,d}(R)$. 
By Lemma \ref{victoryforfields}, the $k$ fiber of $(V,\phi)$ must be a rank $r$ vector bundle on a genus 1, $\Theta$-fold cover of $\P^1_k$.  
We would like to show too that $(V,\phi)$ itself is a rank $r$ bundle on a $\Theta$-fold cover of $\P^1_R$.

To this end, fix
\[ I\inj B\surj A\]
an extension of Artin local rings and assume that the ideal $I$ is square-zero. 
Consider the map
\[\H_{r,d}(B) \ra \H_{r,d}(A)\]
We will show that this map is a torsor for the first hypercohomology of a certain complex (functorial in $I$) 
and that the map surjects onto the kernel of a map from $\H_{r,d}(A)$ to the second hyper cohomology of the same complex:
\[0\ra \Hyp^1(C)\otimes I \ra \H(B) \ra \H(A) \ra \Hyp^2(C)\otimes I\]
In other words, we will compute the complex whose hypercohomology provides the deformation-obstruction theory for $\H_{r,d}$, 
which I will abbreviate to $\H$ for the rest of this section.

We repeat here without proof a commonly known fact about modules over such square-zero extenstions:
\begin{lemma}
  Let $\tilde{X}$ be flat over $\Spec B$ extending $X$ flat over $\Spec A$. Let $\tilde{f}:\tilde{V}\ra \tilde{W}$ be a morphism of $\O_{\tilde{X}}$ bundles extending the map $f:V\ra W$. 
  Then the sheaf of other maps $\tilde{f}':\tilde{V}\ra \tilde{W}$ extending $f$ is a torsor for the additive group $\HOM(V,I\otimes_A W)$.

  In the special case when $\tilde{f}=\id_{\tilde{V}}$, then composition of maps extending $\id_{V}$ corresponds to addition in $I\otimes_A \END(V)$.
  \label{deformmap}
\end{lemma}

Let $(V,\phi)\in \H(A)$ and let $(\tilde{V},\tilde{\phi})$ and $(\tilde{V}',\tilde{\phi}')$ be two lifts to $\H(B)$. 
Let $\U_\bullet$ be a cover that trivializes both so that on $\U_i$ we may choose isomorphisms $f_i:\tilde{V}_i \ra \tilde{V}'_i$. 
We may assume that $f_i$ reduces to the identity modulo $I$, sends $L_i^j$ into $L_i^{\prime j}$ for $p_i\in Z$ and all $j$, 
and that the top wedge power of $f_i$ commutes with the trace isomorphisms identifying the determinants of $\tilde{V}$ and $\tilde{V}'$ with $\O(-1)$.

\begin{defn}
  Let $(V,\phi)\in \H_{r,d}(A)$. The bundle $\END(V)$ has a direct summand $\END_0(V)$ of traceless endomorphisms. 
  By $\END_0'(V)$ we denote the sub-bundle of $\END_0(V)$ which consists of sections preserving the subspaces $L_i^j\sub V|_{p_i}$ for $p_i\in Z$.
  In other words, $\END'_0(V)$ is a specific lower modification of $\END_0(V)$.
  \label{fancyenddefn}
\end{defn}

Now consider the map $f_j^{-1}\circ f_i$ on the double overlap $\U_{ij}$:

\begin{center}
\begin{tikzpicture}[descr/.style={fill=white}]
\matrix(m)[matrix of math nodes,
row sep=4em, column sep=4em,
text height=1.5ex, text depth=0.25ex]
{
  0&I\otimes_A V&\tilde{V}& V &0\\
  0&I\otimes_A V&\tilde{V}& V &0\\
};
\path[->,font=\small]
(m-1-1) edge (m-1-2)
(m-1-2) edge (m-1-3)
(m-1-3) edge (m-1-4)
(m-1-4) edge (m-1-5)
(m-2-1) edge (m-2-2)
(m-2-2) edge (m-2-3)
(m-2-3) edge (m-2-4)
(m-2-4) edge (m-2-5)
(m-1-2) edge node[auto] {$I\otimes\mbox{id}_V$} (m-2-2)
(m-1-3) edge node[below,rotate=90] {$\simeq$} (m-2-3)
(m-1-4) edge node[auto] {$\mbox{id}_V$} (m-2-4);
\end{tikzpicture}
\end{center}

Because $\tilde{V}$ has a canonical filtration by $I\otimes_A V$ and $V$, preserved by $f_i$ and $f_j$, 
this isomorphism may be presented as $\mbox{id}_{\tilde{V}} + h_{ij}$ where $h_{ij}$ is a map $V\ra I\otimes_A V$ linear over $\O_{\P^1_R}$. 
Since the morphism $f_j^{-1}\circ f_i$ preserves the subspaces $L_k$ for $1\leq k\leq 4$ and has determinant 1, we have that 
$h_{\bullet,\bullet}$ is an element of $\check C^1(\U,I\otimes\END_0'(V))$. Here the bundle $\END_0'(V)$ deserves some explanation.

While we're at it, it is convenient to calculate the Serre dual of $\END_0'(V)$.
\begin{lemma}
  Let $\S(V)$ be the sub-bundle of $\Omega'\otimes\END_0(V)$ consisting of the twisted endomorphisms which, in the fiber over $p_i$ for $p_i\in Z$,
  are strictly block upper triangular with respect to the filtration of $V|_{p_i}$ given by $0\sub L_i^j\sub V|_{p_i}$. 
  Then this $\S(V)$ is the Serre dual to $\END_0'(V)$.
  \label{fancyendserredual}
\end{lemma}
\begin{proof}
  The bundle $\END(V)$ is naturally self dual with pairing given by trace of the product of endomorphisms, 
  and $\END_0(V)$ is a direct summand and orthogonal compliment to the scalar endomorphisms, so it too is self dual.
  The bundle $\END_0'$ fits in a sequence of inclusions:
  \[\END_0(V)(-Z)\inj \END_0'(V)\inj \END_0(V)\]
  The quotient of the last bundle by the first is the product \[\prod_{p_i\in Z} \sl(V|_{p_i})\] 
  The bundle $\END_0'$ corresponds to the sub-space of this quotient consisting of traceless endomorphisms of the fibers 
  that are upper triangular with respect to the filtration provided by the $L_i^j$. 
  Taking sheaf hom to $\Omega$, we get:
  \[\Omega\otimes\END_0(V)\inj \Omega\otimes\END_0'^*(V)\inj \Omega'\otimes\END_0(V)\]
  The quotient of the last bundle by the first is naturally the dual of the fiber of $\END_0(V)$, but we use the self duality of $\END_0(V)$ to get rid of this dual. 
  In general, the linear dual of a lower modification is an upper modification of the dual.
  In our case, the condition defining the upper modification requires that the residue endomorphisms are strictly upper triangular with respect to the filtration given by the $L_i^j$.
  This is because these are the endomorphisms of the fiber orthogonal to those which merely preserve the filtration.
\end{proof}

In fact, $h_{\bullet,\bullet}$ is a closed Cech 1-cycle. 
This can be seen by observing $f_k^{-1}\circ f_j \circ f_j^{-1} \circ f_i = f_k^{-1}\circ f_i$ and applying lemma \ref{deformmap}.

Now suppose that $\tilde{V}$ and $\tilde{V}'$ are equipped with Higgs fields $\tilde{\phi}$ and $\tilde{\phi}'$, respectively.
On the cover $\U_\bullet$ we have 

\begin{center}
\begin{tikzpicture}[descr/.style={fill=white}]
\matrix(m)[matrix of math nodes,
row sep=4em, column sep=4em,
text height=1.5ex, text depth=0.25ex]
{
  \tilde{V}&\Omega'\otimes\tilde{V}\\
  \tilde{V}'&\Omega'\otimes\tilde{V}'\\
};
\path[->,font=\small]
(m-1-1) edge node[auto] {$\tilde{\phi}$} (m-1-2)
(m-1-1) edge node[left] {$f_\bullet$} (m-2-1)
(m-1-2) edge node[auto] {$\id_{\Omega'}\otimes f_\bullet$} (m-2-2)
(m-2-1) edge node[auto] {$\tilde{\phi}'$} (m-2-2);
\end{tikzpicture}
\end{center}

Note, this diagram does not have to commute!  
By $f_\bullet^*\tilde{\phi}'$ denote $\id_{\Omega'}\otimes f_\bullet\circ \tilde{\phi}'\circ f_\bullet$ 
and denote by $\delta_\bullet$ the difference $f_\bullet^*\tilde{\phi}' - \tilde{\phi}$.
Taking $(\tilde{V},\tilde{\phi})$ as our basepoint, $(\tilde{V}',\tilde{\phi}')$ is described by the 1-cocycle $h_{\bullet,\bullet}$ 
and the difference $f_\bullet^*\tilde{\phi}' - \tilde{\phi}$ considered as a Higgs field on $\tilde{V}$ restricted to the cover $\U_\bullet$. 
Both $f_\bullet^*\tilde{\phi}'$ and $\tilde{\phi}$ preserve the filtration $L_i^j$ so are well defined on $L_i^j/L_i^{j+1}$.
Both endomorphisms are scalar\textemdash they scale by the $j+1$st from last eigenvalue at the point $p_i$.
So the difference $\delta_\bullet$ both preserves the filtration and vanishes on the successive quotients.
This is the description of the sheaf $\S(V)\iso \Omega\otimes\END'^*(V)$.
Since both Higgs fields are trace zero and agree modulo $I$, we have $\delta_\bullet\in \check C^0(\U,S(V)\otimes_A I)$.

We need to express the requisite coherence between the data $h_{i,j}$ (defining the bundle $\tilde{V}'$ as glued together from $\tilde{V}$) 
and the data $\delta_i = f_i^*\tilde{\phi}' - \tilde{\phi}$. 
We must ensure that the choices $\delta_\bullet$ yield a well defined Higgs field on the bundle $\tilde{V}'$.
This is a statement on double overlaps of $\U_\bullet$.

\begin{center}
\begin{tikzpicture}[descr/.style={fill=white}]
\matrix(m)[matrix of math nodes,
row sep=5em, column sep=5em,
text height=1.5ex, text depth=0.25ex]
{
  (\tilde{V}_i)_j&\Omega'\otimes(\tilde{V}_i)_j\\
  \tilde{V}_{ij}&\Omega'\otimes\tilde{V}_{ij}\\
  (\tilde{V}_j)_i&\Omega'\otimes(\tilde{V}_j)_i\\
};
\path[->,font=\small]
(m-1-1) edge[bend right] node[left] {$1+h_{ij}$} (m-3-1)
(m-1-2) edge[bend left] node[right] {$\id_{\Omega'}\otimes(1+h_{ij})$} (m-3-2)
(m-1-1) edge node[auto] {$\tilde{\phi}_i+\delta_i$} (m-1-2)
(m-1-1) edge node[left] {$f_i$} (m-2-1)
(m-1-2) edge node[auto] {$\id_{\Omega'}\otimes f_i$} (m-2-2)
(m-2-1) edge node[left] {$f_j$} (m-3-1)
(m-2-2) edge node[auto] {$\id_{\Omega'}\otimes f_j$} (m-3-2)
(m-2-1) edge node[auto] {$\tilde{\phi}'$} (m-2-2)
(m-3-1) edge node[auto] {$\tilde{\phi}_j+\delta_j$} (m-3-2);
\end{tikzpicture}
\end{center}

In order for the as yet hypothetical $\tilde{\phi}'$ to be well defined, it is enough to ensure the commutativity of the outside square of the above diagram. 
So the $\delta_\bullet$ and $h_{ij}$ must satisfy the equation
\[(\id_{\Omega'}\otimes(1+h_{ij}))\circ (\tilde{\phi}_i+\delta_i)=(\tilde{\phi}_i+\delta_i)\circ(1+h_{ij})\]
which, using the lemma (\ref{deformmap}), reduces to
\[\tilde{\phi}_i+\id_{\Omega'}\otimes h_{ij}\circ \tilde{\phi_i} + \delta_i = \tilde{\phi}_j+\delta_j + \tilde{\phi}_j\circ h_{ij}\]
Since $\tilde{\phi}$ is a well defined Higgs field on $\tilde{V}$, $\tilde{\phi}_i = \tilde{\phi}_j$.
Now
\[\id_{\Omega'}\otimes h_{ij}\circ \tilde{\phi_i} + \delta_i = \delta_j + \tilde{\phi}_j\circ h_{ij}\]
is an equality of two maps $(\tilde{V}_i)_j\ra \Omega'\otimes(\tilde{V}_i)_j$.
Both maps have image inside $\Omega'\otimes(\tilde{V}_i)_j\otimes_B I$ and hence both are also well defined with domain $(V_i)_j = (\tilde{V}_i)_j/I$. Rewriting we have
\begin{align*}
id_{\Omega'}\otimes h_{ij}\circ \phi_i + \delta_i &= \delta_j + \id_I\otimes \phi_j\circ h_{ij}\\
\delta_i - \delta_j &= \id_I\otimes \phi_j\circ h_{ij} - id_{\Omega'}\otimes h_{ij}\circ \phi_i\\
\delta_i - \delta_j &= [\phi_{ij},h_{ij}]
\end{align*}

What we have are classes $h_{\bullet,\bullet}$ in $\check C^1(\U,\END'_0(V)\otimes I)$ and $\delta_\bullet$ in $\check C^0(\U,\END_0'^*(V)\otimes \Omega\otimes I)$ 
such that $h_{\bullet,\bullet}$ is closed and $d(\delta_\bullet) = [\phi,h_{\bullet,\bullet}]$.
\begin{lemma}
  The set of deformations $(\tilde{V},\tilde{\phi})$ to $(V,\phi)$ is a torsor for the hypercohomology group 
  \[\Hyp^1\left[\END_0'(V)\otimes I \ra \END_0'^*(V)\otimes \Omega\otimes I\right]\]
  where the differential of the complex is given by $[\phi,-]$.
  \label{deformationcomplex}
\end{lemma}

The argument we've just given relied on the existence of an extension $(\tilde{V},\tilde{\phi})$, but it is far from clear that such a thing should exist.  
Trying to prove existence of a deformation we fine:

\begin{lemma}
  There exists a deformation to $(V,\phi)$ iff a certain class in
  \[\Hyp^2(\END_0'(V)\otimes I \ra \END_0'^*(V)\otimes \Omega\otimes I)\]
  vanishes.
  \label{obstructions}
\end{lemma}
\begin{proof}
  Locally on $\P^1_B$, there exist deformations of $V$.  Furthermore, these deformations can be chosen with $\tilde{L}_i^j\sub \tilde{V}|_{p_i}$ extending $L_i^j$.
  Additionally we can pick local isomorphisms to $\O(d)$.
  For example, just pick a cover which trivializes both $V$ and $\O(d)$ and which contains at most one of the $p_i$ each open set.
  Now we can assume that each $\tilde{L}_i^j$ sits as an initial coordinate subspace.  Hence a deformation may be picked in a trivial way.

  Not only are local deformations sure to exist, but they are all equivalent to the deformation described above on a fine enough cover.
  Hence we may pick the trivial local deformations described above on a fine enough cover $\U_\bullet$,
  and the obstruction to glueing the $(\tilde{V}_i,\tilde{\phi}_i)$ into a global deformation will be the obstruction to having any global deformations at all.

  Having chosen the local deformations of the bundle and of the Higgs field, pick 
  \[\psi_{ij}:\tilde{V_i}|_{\U_{ij}}\ra \tilde{V_j}|_{\U_{ij}}\]
  such that
  \begin{enumerate}
    \item $\psi_{ij}$ is the identity on $V$
    \item $\psi_{ij}$ stabilizes all the $\tilde{L}_i^j$.
    \item $\det \psi_{ij} = 1$ with respect to the local isomorphisms $\det \tilde{V_i} \iso \det \tilde{V_j} \iso \O(d)$
  \end{enumerate}
  Such maps are a torsor for the subgroup of $\Aut(\tilde{V_j}|_{\U_{ij}})$ of automorphisms satisfying the three properties required of $\psi_{ij}$.
  Such an automorphism differs from the identity by $\sigma_{ij}$ a section of $\END_0'(\tilde{V_j})\otimes_B I\iso \END_0'(V)\otimes_A I$. 
  In similar fashion, these maps are a torsor for a subgroup of $\Aut(\tilde{V_i}|_{\U_{ij}})$ and in turn can be identified with a section of $\END_0'(V)\otimes_A I$.
  Suppose that we have 
  \begin{align*}
    \psi_{ij}' &= (1+\sigma)\circ \psi_{ij}\\
    &= \psi_{ij}\circ \psi_{ij}\i \circ (1+\sigma) \circ \psi_{ij}
  \end{align*}
  This is $\psi_{ij}$ pre-composed by the automorphism $\psi_{ij}\i \circ (1+\sigma) \circ \psi_{ij}$ of $\tilde{V_i}$. 
  Since $\psi_{ij}$ is $1$ modulo $I$ and since $\sigma$ is annihilated by $I$, this automorphism is still $1+\sigma$.
  So the maps we are interested in are always a torsor for $\END_0'(V)$,
  and it doesn't matter whether this torsorial action is computed on the domain or codomain of the isomorphism $\psi_{i,j}$.

  Fixing our choice of $\psi_{\bullet,\bullet}$, we define
  \[\part_{ijk} = \psi_{ki}\circ \psi_{jk}\circ \psi_{ij}\]
  The maps $\part_{ijk}$ taken in total comprise a two cycle valued in the sheaf $\END_0'(V)\otimes_A I$

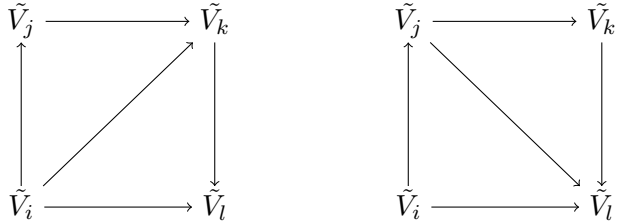
  \begin{center}
  \begin{tikzpicture}[descr/.style={fill=white}]
  \matrix(m)[matrix of math nodes,
  row sep=5em, column sep=5em,
  text height=1.5ex, text depth=0.25ex]
  {
    \tilde{V_j}&\tilde{V_k}&\tilde{V_j}&\tilde{V_k}\\
    \tilde{V_i}&\tilde{V_l}&\tilde{V_i}&\tilde{V_l}\\
  };
  \path[->,font=\small]
  (m-2-1) edge (m-1-1)
  (m-1-1) edge (m-1-2)
  (m-1-2) edge (m-2-2)
  (m-2-1) edge (m-2-2)
  (m-2-1) edge (m-1-2)
  (m-2-3) edge (m-1-3)
  (m-1-3) edge (m-1-4)
  (m-1-4) edge (m-2-4)
  (m-2-3) edge (m-2-4)
  (m-1-3) edge (m-2-4);
  \end{tikzpicture}
  \end{center}
  This is because (consulting the above diagram):
  \begin{align*}
    \part_{ijk}+\part_{ikl} &= \psi_{li}\circ\psi_{kl}\circ\psi_{ik}\circ\psi_{ki}\circ\psi_{jk}\circ\psi_{ij}\\
    &=\psi_{li}\circ\psi_{jl}\circ\psi_{ij}\circ\psi_{ji}\circ\psi_{lj}\circ\psi_{kl}\circ\psi_{jk}\circ\psi_{ij}\\
    &=\psi_{ij}\i \circ\part_{jkl} \circ\psi_{ij}\i + \part_{ijl}
  \end{align*}
  Any other choice $\psi_{\bullet,\bullet}'$ differs by post-composition by $1+\sigma_{ij}$, 
  and the resulting set of $\part_{\bullet,\bullet,\bullet}'$ will differ by the Cech boundary of the $\sigma_{\bullet,\bullet}$.
  The $\psi_{ij}$ provide glueing data iff $\part_{\bullet,\bullet,\bullet} = 0$, hence there exists a consistent choice of glueing data iff $\part$ is of trivial cohomological class.
  Since we work with a curve, this will always be satisfied.

  Now we also pick Higgs fields $\tilde{\phi_i}$ locally extending the Higgs field $\phi$. This choice of 
  \[\tilde{\phi_i}:\tilde{V_i}\ra \Omega'\otimes\tilde{V_i}\]
  is a torsor for the local sections of $\Omega\otimes \END_0'^*(V)$ as computed in the proof of Lemma (\ref{deformationcomplex}).
  \begin{center}
  \begin{tikzpicture}[descr/.style={fill=white}]
  \matrix(m)[matrix of math nodes,
  row sep=5em, column sep=5em,
  text height=1.5ex, text depth=0.25ex]
  {
    \tilde{V_i}&\Omega'\otimes \tilde{V_i}\\
    \tilde{V_j}&\Omega'\otimes \tilde{V_j}\\
  };
  \path[->,font=\small]
  (m-1-1) edge node[auto] {$\phi_i + \delta_i$} (m-1-2)
  (m-2-1) edge node[auto] {$\phi_j + \delta_j$} (m-2-2)
  (m-1-1) edge node[left] {$\psi_{ij}+\sigma_{ij}$} (m-2-1)
  (m-1-2) edge node[auto] {$\psi_{ij}+\sigma_{ij}$} (m-2-2);
  \end{tikzpicture}
  \end{center}
  Consider the commutator $\eta_{ij} = \psi_{ij}\circ \phi_i - \phi_j\circ \psi_{ij}$.
  This is a morphism $\tilde{V_i}|_{\U_{ij}} \ra \Omega'\otimes\tilde{V_j}|_{\U_{ij}}$.
  Following the usual argument, these difference must lie in $\Omega\otimes\END_0'^*(V)\otimes_A I$:
  In the fiber at $p_k\in Z$, both terms act by the same scalar on the successive quotients $L_k^\bullet/L_k^{\bullet+1}$.
  Hence the difference has residues at $p_k$ which are strictly block upper triangular with respect to the flag given by $\tilde{L_k}$.
  The terms also agree modulo $I$, as both are deformations of the Higgs field $\phi$.

  Now suppose that $\psi_{ij}$ is altered by
  \[\sigma_{ij}\in \Gamma(\U_{ij},\END_0'V\otimes_A I)\]
  and that $\phi_i$ is altered by
  \[\delta_i\in \Gamma(\U_i,\Omega\otimes\END_0'^*\otimes_A I)\]
  We compute the resulting change to $\eta_{ij}$:
  \begin{align*}
    (\psi_{ij} + \sigma_{ij})(\phi_i + \delta_i) - (\phi_j+\delta_j)(\psi_{ij}+\sigma_{ij})&=\eta_{ij}+\psi_{ij}\delta_i - \delta_j\psi_{ij} + \sigma_{ij}\phi_i - \phi_j\sigma_{ij}\\
    &= \eta_{ij}+\delta_i - \delta_j + \sigma_{ij}\phi_i - \phi_j\sigma_{ij}
  \end{align*}
  The second equality is because $\psi_{ij}$ is the identity of $V$ modulo $I$. Since both $\phi_i$ and $\phi_j$ are both $\phi$ modulo $I$, we have
  \[\eta'_{ij} = \eta_{ij} + (\delta_i - \delta_j) + [\phi,\sigma_{ij}]\]
  \begin{center}
  \begin{tikzpicture}[descr/.style={fill=white}]
  \matrix(m)[matrix of math nodes,
  row sep=3.5em, column sep=2.5em,
  text height=1.5ex, text depth=0.25ex]
  {
    \check C^3(\END V)&\\
    \part_{\bullet,\bullet,\bullet}\in\check C^2(\END V)&\check C^2(\Omega'\otimes\END V)\\
    \check C^1(\END V)&\check C^1(\Omega'\otimes\END V)\ni\eta_{\bullet,\bullet}\\
    \check C^0(\END V)&\check C^0(\Omega'\otimes\END V)\\
  };
  \path[->,font=\small]
  (m-2-1) edge (m-2-2)
  (m-3-1) edge (m-3-2)
  (m-4-1) edge (m-4-2)
  (m-4-1) edge (m-3-1)
  (m-3-1) edge (m-2-1)
  (m-2-1) edge (m-1-1)
  (m-4-2) edge (m-3-2)
  (m-3-2) edge (m-2-2);
  \end{tikzpicture}
  \end{center}
  We conclude that the obstruction to simultaneously deforming $V$ and $\phi$ lies in the vanishing of the hyper cohomology class $
  (\part_{\bullet,\bullet,\bullet},\eta_{\bullet,\bullet})$ if we can ensure that the pair is actually closed in hyper cohomology. 
  We have already seen that $\part_{\bullet,\bullet,\bullet}$ is a cocycle in the cohomology of $\END_0'(V)$, 
  so all that remains to check is that $[\phi,\part_{\bullet,\bullet,\bullet}] = d(\eta_{ij})$.

  Consider a triple overlap $\U_{ijk}$. Here the element $\eta_{ij}$ can be written $\psi_{jk}\psi_{ij}\phi_i - \psi_{jk}\phi_j\psi_{ij}$ because $\psi_{jk}$ is $1$ modulo $I$.
  Likewise $\eta_{jk} = \psi_{jk}\phi_j\psi_{ij} - \phi_k\psi_{jk}\psi_{ij}$ and $\eta_{ik} = \psi_{ik}\phi_i - \phi_k\psi_{ik}$.
  So we compute the Cech boundary
  \begin{align*}
    d(\eta_{\bullet,\bullet}|_{\U_{ijk}}) &= \eta_{ij} + \eta_{jk} - \eta_{ik}\\
    &= \psi_{jk}\psi_{ij}\phi_i - \psi_{jk}\phi_j\psi_{ij} + \psi_{jk}\phi_j\psi_{ij} - \phi_k\psi_{jk}\psi_{ij} - \psi_{ik}\phi_i + \phi_k\psi_{ik}\\
    &= \psi_{jk}\psi_{ij}\phi_i - \phi_k\psi_{jk}\psi_{ij} - \psi_{ik}\phi_i + \phi_k\psi_{ik}\\
    &= (\psi_{jk}\psi_{ij} - \psi_{ik})\phi_i - \phi_k(\psi_{jk}\psi_{ij} - \psi_{ik})\\
    &= [\phi,\psi_{jk}\psi_{ij} - \psi_{ik}]\\
    &= [\phi,\part_{ijk}]
  \end{align*}
  The second to last equality above is justified because $\phi_i$ and $\phi_k$ are $\phi$ modulo $I$.

  Thus we can rightly say that all the choices for $\psi_{\bullet,\bullet}$ and $\phi_\bullet$ can be made in a consistent way leading to a globally defined deformation of $(V,\phi)$ 
  iff the class $(\part,\eta)$ in hypercohomology vanishes.
\end{proof}

Denote by $C$ the complex \[C = \END_0'(V) \ra \END_0'^*(V)\otimes \Omega\] with differential given by $[\phi,-]$. 
The first hypercohomology of the complex $C\otimes_A I$ describes the torsor of deformations, 
and the second hypercohomology of the same is the space of obstructions to existence of a deformation.
Hence we are interested in the following cohomological result:
\begin{lemma}
  The hypercohomology of $C$ is concentrated in degree one and is a free $A$ module of rank two.
  \label{computedeformations}
\end{lemma}
\begin{proof}
  We'll apply proper base change. Let $x$ be the closed point of the scheme $\Spec A$. 
  Since $C$ is composed of vector bundles on $\P^1_A$, the restriction to $\P^1_x$ is underived
  and we turn to computing the hypercohomology in the case $A=\kappa(x)$, $I = \kappa(x)^1$.

  We need to compute the Euler characteristic of the sheaf $\END_0' V$:
  \begin{center}
  \begin{tikzpicture}[descr/.style={fill=white}]
  \matrix(m)[matrix of math nodes,
  row sep=2.5em, column sep=2.0em,
  text height=1.5ex, text depth=0.25ex]
  {
    0&0&0\\
    \O_{\P^1}(-Z)&\O_{\P^1}&\O_{\P^1}\\
    \END V (-Z)&\END' V&\END V\\
    \END_0 V (-Z)&\END_0' V&\END_0 V\\
    0&0&0\\
  };
  \path[right hook->,font=\small]
  (m-2-1) edge (m-2-2)
  (m-2-2) edge (m-2-3)
  (m-3-1) edge (m-3-2)
  (m-3-2) edge (m-3-3)
  (m-4-1) edge (m-4-2)
  (m-4-2) edge (m-4-3);
  \path[->,font=\small]
  (m-2-2) edge (m-1-2)
  (m-2-3) edge (m-1-3)
  (m-2-1) edge (m-1-1)
  (m-3-2) edge (m-2-2)
  (m-3-3) edge (m-2-3)
  (m-3-1) edge (m-2-1)
  (m-4-2) edge (m-3-2)
  (m-4-3) edge (m-3-3)
  (m-4-1) edge (m-3-1)
  (m-5-2) edge (m-4-2)
  (m-5-3) edge (m-4-3)
  (m-5-1) edge (m-4-1);
  \end{tikzpicture}
  \end{center}
  The columns in the above diagram are exact, and the rows are inclusions.
  The bundle $\END V$ is degree zero and rank $\Theta^2 r^2$ on $\P^1$ hence has Euler characteristic $\Theta^2 r^2$.
  The sub bundle $\END' V$ loses Euler characteristic due to preserving the parabolic structures at the points of $Z$, and one can calculate $\chi(\END' V) = 0$.
  Since the middle column is exact, the Euler characteristic of our bundle of interest, $\END_0' V$, is $-1$.
  We set $h^0(\END_0' V) = e$, $h^1(\END_0' V) = e+1$.
  Since $\END_0^{\prime *} (V)\otimes\Omega$ is the Serre dual of $\END_0' V$, we have $h^0(\END_0^{\prime *} (V)\otimes\Omega)=d+1$, $h^1(\END_0^{\prime *} (V)\otimes\Omega)=d$.

  Suppose now that $\kappa(x)$ is an algebraically closed field.
  Then since $(V,\phi)$ is stable, the endomorphisms of $V$ commuting with $\phi$ form a division algebra over $\kappa(x)$ and hence are only scalars from $\kappa(x)$.
  This means that the map $[\phi,-]:H^0(\END_0' V)\ra H^0(\END_0^{\prime *} (V))\otimes \Omega$ is rank $e$.
  The Serre dual of this map $H^1(\END_0^{\prime *}(V)\otimes \Omega)\ra H^1(\END_0' (V))$ is also rank $e$.
  This is enough to decide that the hypercohomology of $C$ is concentrated in degree one, of dimension two.
  
  If $\kappa(x)$ is not algebraically closed, take $K$ an algebraic closure and let $(V_K,\phi_K)$ be the base change.
  Since we require that $(V_K,\phi_K)$ is stable, the hypercohomology of $C$ is concentrated in degree one, of dimension two.
  Hence we conclude the same for the deformation theory over $\kappa(x)$.

  Return to the general case of an artinian ring $A$ and square zero extension defined by $A^1$.
  Now $\Hyp^\bullet(C)$ is a complex in $D(A)$ concentrated in degrees zero, one, and two.
  Applying Lemma \ref{nicelemma}, we find that the hypercohomology of $C$ is a free $A$ module of rank two concentrated in degree one.
\end{proof}

\section{Fourier-Mukai Transforms}

The main goal of this section is to show that $H_{r,-1}$ is independent of the parameter $r$.
In fact, we will show a stronger statement which works for a varieties of values of the parameter $d$.
To do this, we first work with the stacks $\E_{r,d}$, which are naturally set up for the integral transforms in \cite{spanish}.
In that paper, the authors expand on Atiyah's result on vector bundles over elliptic curves
by describing a method of producing isomorphisms between moduli spaces of vector bundles on genus one curves with Gorenstein singularities.
By Lemma \ref{arithmeticgenus}, $N_\phi$ is supported on a Gorenstein curve of arithmetic genus one.

In the notation of \cite{spanish}, all of our modules $M_\phi$ are pure of dimension one.
This is because any subsheaf with zero dimensional support would indicate torsion in the underlying $\O_{\P^1}$ module, which is impossible.
Slope, and so stability, are defined in \cite{spanish} in terms of the Hilbert polynomials.
In order to proceed, we must choose an ample line bundle for our spectral curves, and so we pick $\pi^*(\O_{\P^1}(1))$.
Then the Hilbert polynomial of our module $M_\phi$ is $\Theta r t + \Theta r - 1$.
Hence the rank and degree of $M_\phi$ are $\Theta r$ and $\Theta r-1$ in the sense of \cite{spanish} and
the fact that we assume $(V,\phi)$ is stable implies that $M_\phi$ is stable in their sense as well.
Applying their techniques, we will prove the following:

\begin{thm}
  Let $(r,d)$ be a pair of integers with $r>0$ and assume $\gcd(\Theta r, d) = 1$.
  Assume further that $[\Theta r, \Theta r + d]^T$ is in the orbit of $[\Theta,1]^T$ under the action of the matrices:
  \[ 
    A_\psi=
    \begin{bmatrix}
      1&0\\
      1&1
    \end{bmatrix}
    ,\qquad
    A_\phi=
    \begin{bmatrix}
      1&-\Theta\\
      0&1
    \end{bmatrix}
    ,\qquad
    -1=
    -
    \begin{bmatrix}
      1&0\\
      0&1
    \end{bmatrix}
  \]
  For example, the case $(r, d)$ with $d\equiv \pm 1$ modulo $\Theta r$.

  Then the stack $\E^\circ_{r,d}$ is equivalent to the stack $\E^\circ_{1,-1}$.
  Further, we can reduce a $T$ family of either of these objects to a line bundle of the graph of a map $T\ra B$.
  As a consequence we find that the coarse moduli space of any such $\E_{r,d}$ is given by $B$
  and that $\E^\circ_{r,d}\isom \H^\circ_{r,d}$.
  \label{bigtheorem}
\end{thm}

In order to make this argument in families, we need to study a bundle transformed under some integral transforms by analyzing its derived fibers.
It will be most convenient to have the following lemma handy for the proof of the theorem.

\begin{lemma}
  Let $\pi:E\ra T$ be a quasi-projective, flat morphism of Noetherian schemes and let $A$ be an object of $D^-(E)$.
  Assume that for every point $p\in T$, the derived pullback (which exists because of $T$-local existence of enough vector bundles) of $A$ to $\pi\i(p)$ is a sheaf in degree zero.
  Then $A$ is a sheaf in degree zero as well, and is flat over $T$.
  \label{nicelemma}
\end{lemma}
\begin{proof}
  Write $A^i$ for the $i$-th cohomology sheaf of $A$.
  We have a spectral sequence for restriction to the fiber $E_p = \pi\i(p)$ in terms of the cohomology sheaves. 
  Denote by $u$ the inclusion map from $E_p$ to $E$.
  This spectral sequence has the group $L_i u^*(A^j)$ at the $(-i,j)$ spot on the second page.
  If there is a positive $j$ with $A^j\neq 0$, then there is a greatest such $j$.
  By Nakayama's Lemma, the term at $(0,j)$ must be non-zero for some $p$, which contradicts the assumption that all derived fibers are concentrated in degree zero.
  So our page two has terms only at non-positive coordinates.

  We'd like to show that $L_i u^*(A^j)$ is zero for $j = 0 $ and $i>0$.
  We start with the case $i=1$.
  This term must be zero as there is nothing in the $(-3-k,1+k)$ or $(1+k,-1-k)$ spots for the rest of eternity ($k\geq 0$).
  Then
  \[L_1 u^*(A^0) =0 \]

  Because the functors $L_i u^*$ are calculated by resolutions by vector bundles, we can relate the stalks of these derivatives to usual constructions in homological algebra.
  Specifically, for any $q\in E_p$, if we take the stalk at $q$ of $L_1 u^*(A^0)$, we find that
  \[\Tor_1^{\O_{E,q}}(\O_{E_p,q},A^0_q) = 0\]
  Since $\O_{E_p,q} \isom \O_{E,q}\otimes_{\O_{T,p}}\kappa(p)$, we have the following isomorphism of functors:
  \[(-\otimes_{\O_{E,q}}(\O_{E,q}\otimes_{\O_{T,p}}\kappa(p)))\isom (-\otimes_{\O_{T,p}}\kappa(p))\]
  If we restrict scalars and so consider $(A^0_q\otimes_{\O_{T,p}}\kappa(p))$ to be a module over $\O_{T,p}$,
  we find that
  \[\Tor_1^{\O_{T,p}}(\kappa(p),A^0_q) = 0\]
  Now $A^0_q$ is a finite module over $\O_{E,q}$ and $\kappa(p)$ is a field, so our situation satisfies the hypothesis of (\cite{ega1}, 10.2.2),
  and we see that $A^0_q$ is flat over $\O_{T,p}$.
  Since $A^0_q$ is flat over $T$, all higher Tor groups are also zero, as desired.

  In particular $\Tor_2(E_p,A^0) = 0$.
  This implies that the term $\Tor_0(E_p,A^{-1})$ must already be converged\textemdash to zero.
  If $A^{-1}$ were non-zero, this could not be true for every $p$ by Nakayama's lemma, hence $A^{-1} =0$.
  Similarly for all lower cohomologies of $A$, as desired.
\end{proof}

Having established this, we return to the proof of Theorem \ref{bigtheorem}.

\begin{proof}
  Let $(f,N)$ be a family from $\E_{r,d}(T)$. So $f$ is a map $T\ra \A^1$ and $N$ is a module as in definition \ref{substackdefn2} on $E=E_f=T\times_{\A^1}B$.
  We'll suppress notation and write $E$ for $E_f$.
  Note that, as a result of the discrepancy between degree and Euler characteristic on $\P^1$, the fiber of an object in $\E_{r,d}$ is a bundle of degree $\Theta r + d$.
  Also, due to the way rank is measured on these singular curves, the module is rank $\Theta r$ on $E$.
  We are interested in two Fourier-Mukai transforms from $E\ra E$.

  The first one has kernel $\pi^*\O(1)$ on the diagonal of $E\times_T E$.
  The corresponding transform on the category $D^b(E)$ is to tensor by the line bundle $\pi^*\O(1)$, for convenience we will call this operator $\Psi$.
  This gives an easy equivalence $\E^\circ_{r,d} \isom \E^\circ_{r,d+r}$ which corresponds to the action of $A_\psi$ on the vector $[\Theta r, \Theta r + d]^T$.
  
  The second operation is more complicated.
  Let $I_\Delta$ be the ideal sheaf of the diagonal $E\xrightarrow{\Delta}E\times_T E$.
  This yields a Fourier-Mukai equivalence on $E$ which we will call $\Phi$.
  The kernel $I_\Delta$ is finite homological dimension over $E$ as it forms an exact triangle with the kernels $\O_\Delta$ and $\O_{E\times E}$,
  which both have finite homological dimension.
  This means that $\Phi$ maps $D^b(E)\ra D^b(E)$.
  
  To compute the results of this transform in families, we would like to compare to the same transform on fibers over points of $T$.
  Let $A$ be an object in the bounded derived category of $E$. 
  If we compute the derived fiber of $\Psi(A)$ at a point $t\in T$, it is the same as taking the derived fiber of $A$ and twisting by $\pi^*\O(1)$.
  This is simply because the pullback functor to the fiber is monoidal.
  
  What can we say about the fiber of $\Phi(A)$?
  Our candidate is the Fourier-Mukai operator on the fiber of $E_t$ with kernel the ideal sheaf of the diagonal in $E_t\times E_t$, abuse notation and call this $\Phi_t$.
  The operator $\Phi$ is composed of three steps, and between each of these stages we may restrict to fibers to compare with $\Phi_t$.
  The first stage in both $\Phi$ and $\Phi_t$ is a derived pullback, so it certainly commutes with restriction to fibers.
  The third stage in both operators is a pushforward, and these will commute with derived restriction to fibers by base change.
  The second stage, tensoring by the kernel, requires that we compute the derived fiber of the kernel $I_\Delta$.
  The underived fiber of $I_\Delta$ is the ideal sheaf of the diagonal in $E_t\times E_t$.
  $I_\Delta$ is the kernel of the surjective map $\O_{E\times E}\ra \O_E$, and since $E\ra T$ is flat, the flatness of $I_\Delta$ over $T$ follows.
  Since $I_\Delta$ is flat over $T$, its derived fiber at $t\in T$ agrees with its underived fiber.
  Hence the fiber of $\Phi$ at $t$ is what we have called $\Phi_t$.
  
  The module $N$ is, fiberwise, a sheaf of degree $\Theta r+d$ and rank $\Theta r$ due to our choice of ample bundle.
  Let $t\in T$ be an arbitrary point.
  Consulting \cite[Prop. 1.9, 1.13]{spanish}, the fiber $\Phi(N)|_t$ is a semi-stable sheaf of rank $\Theta r - \Theta (\Theta r + d)$ and degree $\Theta r + d$ concentrated in degree $1$
  if $\Theta r + d \leq r$.
  If the opposite is true then the fiber is a semi-stable sheaf of rank $\Theta(\Theta r + d) - \Theta r$ and degree $-\Theta r - d$ concentrated in degree $0$.
  In this former case, we will abuse notation and write $\Phi(N)$ for the first cohomology sheaf of $\Phi(N)$.
  In either case, $\Phi$ will induce an equivalence $\E^\circ_{r,d}\ra \E^\circ_{r',d'}$ where $r', d'$ are the parameters for $\E$ corresponding to sheaves of degree and rank
  \[ [\Theta r ,\Theta r + d]^T\cdot A_\phi\]
  Explicitly: $r' = r - (\Theta r + d)$ and $d' = d - \Theta(\Theta r + d)$.
  Before declaring victory, we must check that the transformed sheaf is balanced in the sense that it has the expected length
  when restricted to the exceptional divisors $E_{i,j}$.
  The is easily checked once one has the exact triangle provided on \cite[p. 8]{spanish}:
  \[\Phi(\mathcal{E})[-1]\ra R\Gamma(\mathcal{E})\otimes \O_{S'} \ra \mathcal{E} \ra \Phi(\mathcal{E})[1]\]
  which we apply here to $\mathcal{E} = (N)$.
  Knowing that $\phi(\mathcal{E})$  and $\mathcal{E}$ are sheaves concentrated in degree zero or one and that $\chi(R\Gamma(\mathcal{E})) = -1$,
  we can deduce that $\mathcal{E}$ is dimension $c_ir$ when restricted to the divisor $E_{i,j}$ iff $\Phi(\mathcal{E})$ is dimension $c_i(r')$ when restricted to the same divisor.
  This latter fact ensures that our generalized vector bundles remain balanced rank on different components of $S'$, should they exist.

  Furthermore, $\Phi(N)|_t$ is a sheaf pure of dimension one, so flat over $\P^1_t$, 
  which implies that the derived fiber to a point $p\in \P^1_t$ is a vector space in degree 0.
  Applying Lemma \ref{nicelemma} to the sheaf $\Phi(N)$ and the morphism $E_f\ra \P^1_T$
  we can conclude that the transformation of the sheaf is still a sheaf in degree 0, flat over $\P^1_T$, and with the right numerical invariants to be an object of $\E_{r',d'}$.

  Inducting on $r$, we have functors connecting the stacks $\E^\circ_{r,d}\isom \E^\circ_{1,1}$ in all four cases I through IV.
  The equivalences just described in terms of integral transforms cannot produce isomorphisms of the Deligne-Mumford stacks $\E_{r,d}$.
  This is because  a trivialization of the determinant line bundle of $V$ does not yield a trivialization of the determinant of its transform.

  At any rate, the integral transforms described above are enough to yield isomorphisms of coarse spaces.
  Applying once more the integral transform $\Phi\i$, we obtain a sheaf $\F$ on $E_f$ which is flat over $T$ 
  and which, in every fiber $E_t$, is isomorphic to the structure sheaf of a point.
  The pushforward of $\F$ to $T$ is a line bundle, so locally on $T$ we may trivialize $\F$ so it is the structure sheaf of a graph of a map $T\ra B$.
  Hence the Zariski sheafification of the presheaf of isomorphism classes of $\E_{r,d}$ is represented by the scheme $B$.

  Since $B$ is smooth, this tells us that the stack $\E_{r,d}$ has an unobstructed deformation theory given by a free module of rank two.
  Recall (Lemma \ref{computedeformations}) that the deformation theory of $\H_{r,d}$ is also two dimensional and unobstructed.
  We have a map of coarse moduli of stacks $\E_{r,d}\ra\H_{r,d}$ and a corresponding map of coarse spaces $E_{r,d}\ra H_{r,d}$
  By Lemma \ref{victoryforfields}, this coarse space map is an isomorphism on points.
  Since we now know that $\E$ and $\H$ have the same deformation/obstruction theory, the map is an isomorphism on tangent vectors,
  hence an isomorphism.

  We would still like to upgrade this to show that $\E\ra \H$ is an isomorphism of stacks.
  To any family $(V,\phi)$, we have the associated module $N_\phi$ on $B\times T$ and we would like to show that $N_\phi$ is supported on a divisor of type $\Delta$.
  The characteristic polynomial of $\phi$ is a section on $|\Omega'|$ of the line bundle $\pi^*\Omega^{\prime \otimes \Theta r}$.
  The top term of $P_\phi$ is $y^{\Theta r}$, so as a section on $\P(\Omega'\oplus \O)$ it has a pole of order $\Theta r$ at the infinity section,
  hence we regard $\P_\phi$ as a section of $\pi^*\Omega^{\prime \otimes \Theta r}(\Theta r\ \infty)$.
  Since the strict transform $S'$ deducts $c_{i,j} r$ copies of the exceptional divisor $E_{i,j}$ from the pullback of $S$,
  the equation for $S'$ is a section of $\O(r\Delta)$.
  We have the equation of the strict transform 
  \[P'_\phi\in \Gamma(W\times T,\O(r\Delta)) \isom \Gamma(\P^1_T, \O(r)) \isom \Sym^r \Gamma(\P^1_T,\O(1))\]
  Since the subscheme $S'$ doesn't intersect the divisor at infinity of $W\times T$, we may pick $x$ a linear function on $P^1$ not vanishing at infinity 
  and set $P'_\phi = f(x)$ for $f$ a monic polynomial of degree $r$.
  Say $f(x) = x^r + a x^{r-1}\cdots$. Set $g = (x+\frac{a}{r})$. The only way $f$ can be an $r$th power is if $f=g^r$, which we will ultimately show.

  In the fiber over any point of $T$, Lemma \ref{victoryforfields} guarantees that $g$ annihilates the module $N_\phi$.
  All we want to show is that $g$ annihilates the whole module $N_\phi$ over $T$, so without loss of generality, assume $T$ is affine with coordinate ring $A$.
  Let $Q_i$ form a primary decomposition of the zero ideal in $A$ and set $P_i=\sqrt{Q_i}$.
  For every $i$, we have the Artin local ring $(A_{P_i}/Q_i, P_i)$.
  The Higgs bundle at the closed point of this Artin ring is in the image of $\E_{r,d}$, and because the map $\E_{r,d}\ra \H_{r,d}$
  induces an isomorphism of deformation obstruction theories, the Higgs bundle on the entire ring $A_{P_i}/Q_i$ must also be in the image of $\E$.
  This means that, restricted over the Artin ring, $f = g^r$ and $g$ annihilates $N_\phi$.
  Since the $Q_i$ form a primary decomposition of $0$, $g$ must annihilate $N_\phi$ over the whole ring $A$.
  So, to the module $N_\phi$ associated to $(V,\phi)$ is supported on a curve of type $\Delta$, as in the definition of $\E$, 
  and hence $(V,\phi)$ is in the image of the stack $\E_{r,}$.
\end{proof}

\pagebreak

\bibliographystyle{plain}
\bibliography{higgs}
\end{document}